\documentclass[12pt,a4paper]{article}
  
\usepackage{amsmath,amssymb} 
\usepackage{graphicx,color} 
\usepackage{euscript,upref} 
\usepackage{url} 
  
\newcounter{defi}
\newcounter{theo}
\newcounter{propo} 
\newcounter{examp} 

\newcounter{MatrixTheo} 
\newcounter{AbsRegTheo}
\newcounter{AbsRegDefi}
\newcounter{SplitDefi} 
\newcounter{ARMSplitTheo} 
\newcounter{ImmersDefi} 
\newcommand{\mbf}[1]{\textbf{\textit{#1}}}
 
\newcommand{\mcl}{\mathcal} 
\newcommand{\mbb}{\mathbb} 
 
\newcommand{\eus}{\EuScript} 
\newcommand{\diag}{\mathrm{diag}\,} 
\newcommand{\dist}{\mathrm{dist}\;} 
\newcommand{\Dist}{\mathrm{Dist}\,}
\newcommand{\ov}{\overline} 
\newcommand{\un}{\underline} 
\newcommand{\m}{\mathrm{mid}\;} 
 
\newcommand{\sgn}{{\mathrm{sgn}\;}} 
\newcommand{\pro}{{\mathrm{pro}\;}} 
\newcommand{\opp}{{\mathrm{opp}\;}} 
\newcommand{\inv}{{\mathrm{inv}\,}}
\newcommand{\sti}{{\mathrm{sti}\;}} 
\newcommand{\dual}{{\mathrm{dual}\;}}

\newcommand{\tl}{\widetilde{\phantom{p}}}
\newcommand{\ii}[2]{\hspace*{-\arraycolsep}
\begin{array}[t]{c}\\[-15pt]\mbox{\unitlength=1mm%
\begin{picture}(5,5)(-1,0)\linethickness{1.6pt}
\put(0,5){\line(0,-1){5}} \put(3,5){\line(0,-1){5}}
\put(0.2,0){\line(3,5){3}}\put(0.1,0){\line(3,5){3}}
\put(0,0){\line(3,5){3}}  \put(0.1,0){\line(3,5){3}}
\put(0.2,0){\line(3,5){3}}\end{picture}}^{\scriptstyle#1}\\[-7pt]
{\scriptstyle#2\hspace{1ex}}\end{array}\hspace{-0.5\arraycolsep}}
  
\renewcommand{\r}{\mathrm{rad}\;} 
\definecolor{MyGreen}{rgb}{0.1,0.7,0.2} 
\definecolor{Gray1}{rgb}{0.6,0.6,0.65}
\definecolor{Gray2}{rgb}{0.4,0.45,0.4}
\definecolor{Gray3}{rgb}{0.6,0.55,0.55}
  
\newenvironment{definition}{\addvspace{\bigskipamount} 
  \noindent{\bf Definition~\arabic{defi}}\it}{\addtocounter{defi}{1} 
  \par\addvspace{\bigskipamount}} 
\newenvironment{theorem}{\addvspace{\bigskipamount} 
  \noindent{\bf Theorem~\arabic{theo}}\it}{\addtocounter{theo}{1} 
  \par\addvspace{\bigskipamount}} 
\newenvironment{proposition}{\addvspace{\bigskipamount} 
  \noindent{\bf Proposition~\arabic{propo}}\it}{\addtocounter{propo}{1} 
  \par\addvspace{\bigskipamount}} 
\newenvironment{proof}{\addvspace{\bigskipamount} 
  \noindent{\sl Proof.}}{\hfill$\blacksquare$\par\addvspace{\bigskipamount}} 
\newenvironment{corollary}{\addvspace{\bigskipamount} 
  \noindent{\bf Corollary}}{\par\addvspace{\bigskipamount}} 
\newenvironment{example}{\addvspace{3ex}\noindent%
  {\textbf{Example~\arabic{examp}\;}}}%
  {\addtocounter{examp}{1}\hfill$\blacksquare$\par\addvspace{3ex}}
  
\title{\LARGE\bf Numerical computation \\[1mm]
                 of formal solutions to interval \\[1mm]
                 linear systems of equations }
\author{\sc Sergey P. Shary}
\date{\sl\small Institute of Computational Technologies SB RAS\\
                and Novosibirsk State University, \\
                Novosibirsk, Russia \\
                E-mail: \texttt{shary@ict.nsc.ru}}

\begin{document}
\maketitle
  
\begin{abstract} 
The work is devoted to the development of numerical methods for computing \emph{formal 
solutions} of interval systems of linear algebraic equations $\mbf{A}x = \mbf{b}$. 
These solutions are found in Kaucher interval arithmetic, which extends and completes 
the classical interval arithmetic algebraically. The need to solve these problems naturally 
arises, for example, in inner and outer estimation of various solution sets to interval 
linear systems of equations. The work develops two approaches to the construction of 
stationary iterative methods for computing the formal solutions that are based on splitting 
the matrix of the system. We consider their convergence and implementation issues, compare 
with the other approaches to computing formal solutions. 
\end{abstract}
  
  
\section{Introduction}

The main object we study in this article is interval systems of linear algebraic 
equations having the form 
\begin{equation}
\label{InLinSys1}
\arraycolsep 1mm
\left\{
\begin{array}{ccccccccc}
\mbf{a}_{11} x_1 &+& \mbf{a}_{12} x_2 &+&\ldots&+&
\mbf{a}_{1n} x_n &=& \mbf{b}_1, \\[2mm]
\mbf{a}_{21} x_1 &+& \mbf{a}_{22} x_2 &+&\ldots&+&
\mbf{a}_{2n} x_n &=& \mbf{b}_2, \\[2mm]
& \vdots &&& \ddots &&& \vdots & \\[2mm]
\mbf{a}_{n1} x_1 &+& \mbf{a}_{n2} x_2 &+&\ldots&+&
\mbf{a}_{nn} x_n &=& \mbf{b}_n, 
\end{array}
\right.
\end{equation} 
with intervals $\mbf{a}_{ij}, \mbf{b}_i$, which will be considered as elements of 
Kaucher complete interval arithmetic $\mbb{KR}$ \cite{Kaucher-80}. It is convenient 
to denote such systems as 
\begin{equation}
\label{InLinSys2}
\mbf{A}x = \mbf{b},
\end{equation}
where $\mbf{A} = (\,\mbf{a}_{ij})$ is an interval $n\times n$-matrix, $\mbf{b} = 
(\,\mbf{b}_{i})$ is an interval $n$-vector. 
  
For interval equations and systems of equations, various definitions of solutions 
and solution sets exist. Additionally, practice requires that these solution sets be 
evaluated in a wide variety of ways. As a consequence, there are many problem statements 
associated with interval equations and systems of equations. In our paper, we are going 
to concentrate on computing the so-called formal solutions to interval linear systems 
of the form \eqref{InLinSys1}--\eqref{InLinSys2}.

\begin{definition} 
An interval vector is called \textsl{formal solution} to an interval equation or a system 
of equations, if substituting it into the equation or equation system and execution 
of all operations in interval arithmetic result in a valid equality. 
\end{definition}
  
Due to a number of reasons, it is useful to find formal solutions to interval systems 
of equations in Kaucher interval arithmetic $\mbb{KR}$ rather than in classical interval 
arithmetic $\mbb{IR}$. $\mbb{KR}$ is an algebraic and order completion of $\mbb{IR}$. 
  
The problem we consider is not new. It was first noted in the paper by S.\,Berti 
\cite{Berti} with respect to one interval quadratic equation. Then H.\,Ratschek and 
W.\,Sauer \cite{RatschSau} studied such solutions for a single interval linear 
equation, and they used the term \emph{algebraic solution}. K.\,Nickel \cite{Nickel} 
considered formal solutions to interval linear systems of equations in complex interval 
arithmetics. However, formal (algebraic) solutions of interval equations and systems 
of equations have been studied only in about a dozen papers for the last decades when 
interval analysis rapidly developed. 
  
The need to consider formal solutions arises, in particular, in the interval analogue 
of the method of undetermined coefficients, as well as in modeling linear static 
systems under interval uncertainty of their parameters, as a consequence of the 
so-called \emph{formal (algebraic) approach} to their analysis \cite{Shary-95, Shary-96,
SharySurvey}. We remind that this approach replaces the original problem of estimating 
the internal states of the system to the problem of finding algebraic solutions of 
an auxiliary equation in Kaucher arithmetic $\mbb{KR}$. In this work, we do not focus 
on the derivation of Kaucher interval arithmetic and discussion of its remarkable 
properties. The interested reader should refer to the original works of 
\cite{GardTre-80, GardTre-82,Kaucher-80}, or to the summary in \cite{Kupriyanova, 
Shary-95, Shary-96,Shary-98, SharySurvey}. 
  
Everywhere below, we will denote the intervals and interval objects (vectors, matrices) 
in bold (for example, $\mbf{A}$, $\mbf{B}$, $\mbf{C}$, \ldots, $\mbf{x}$, $\mbf{y}$, 
$\mbf{z}$), $\un{\mbf{x}}$, $\ov{\mbf{x}}$ are understood as the lower (left) and upper 
(right) endpoints of an interval $\mbf{x}$, other symbols also follow the informal 
international standard \cite{INotation}. Some results of this work were previously 
published in \cite{Shary-95}. 
  
  
\section{Kaucher interval arithmetic}

The classical interval arithmetic $\mbb{IR}$ is known to be an algebraic system 
formed by intervals of the real axis $\mbb{R}$ with the operations between them 
defined ``by representatives'', i.\,e. according to the following fundamental 
principle 
\begin{equation}
\label{FundamProperty}
\mbf{a}\star\mbf{b} \  = \,\{\,a\star b \mid a\in\mbf{a},\,b\in\mbf{b}\,\}. 
\end{equation} 
The constructive reformulation of the above relation for separate arithmetic 
operations looks as follows:  
\begin{align}
& \mbf{a} + \mbf{b} = \left[\,\un{\mbf{a}} + \un{\mbf{b}},\,\ov{\mbf{a}}
 +\ov{\mbf{b}}\,\right],\label{Addition}\\[5pt]
& \mbf{a} - \mbf{b} = \left[\,\un{\mbf{a}} - \ov{\mbf{b}},\,\ov{\mbf{a}}
 - \un{\mbf{b}}\,\right],\label{Subtraction}\\[5pt]
& \mbf{a}\cdot\mbf{b} = \left[\,\min\{\un{\mbf{a}}\,\un{\mbf{b}},
 \un{\mbf{a}}\,\ov{\mbf{b}},\ov{\mbf{a}}\,\un{\mbf{b}},\ov{\mbf{a}}\,
 \ov{\mbf{b}}\}\right.,\left.\max\{\un{\mbf{a}}\,\un{\mbf{b}},
 \un{\mbf{a}}\,\ov{\mbf{b}},\ov{\mbf{a}}\,\un{\mbf{b}},\ov{\mbf{a}}\,
 \ov{\mbf{b}}\}\,\right],\label{Multiplication}\\[5pt]
& \mbf{a}/\mbf{b} = \mbf{a}\cdot\left[\,1/\ov{\mbf{b}},\,1/\un{\mbf{b}}
 \,\right]\qquad\mbox{ for } \ \mbf{b}\not\ni 0.\label{Division}
\end{align}
  
The algebraic properties of the classical interval arithmetic $\mbb{IR}$ are much more 
poor than those of the usual number systems, the ring $\mbb{Z}$ of integers, the fields 
of rational numbers $\mbb{Q}$ and real numbers $\mbb{R}$, because 
\begin{itemize}
\item 
all intervals with nonzero width, i.\,e. most of the elements of $\mbb{IR}$ do not have 
inverse elements with respect to the operations \eqref{Addition}--\eqref{Division},
\item 
the arithmetic operations \eqref{Addition}--\eqref{Division} are related to each other 
by a weak sub-distributivity relation (see \cite{MooreBakerCloud,Neumaier,SharyBook}), 
and the full distributivity of multiplication (and division) with respect to addition 
and subtraction does not take place. 
\end{itemize} 
As a consequence, first, in $\mbb{IR}$ elementary equations with respect to the unknown 
variable $x$
\begin{equation*}
\label{SimpleEqs}
\mbf{a} + x = \mbf{b},\hspace{34mm}\mbf{a}\cdot x = \mbf{b}
\end{equation*}
and their like do not always have solutions. Secondly, the technique of symbolic 
transformations in the classical interval arithmetic $\mbb{IR}$ is quite poor. 
We cannot even transfer terms from one part of the equation to another, and the lack 
of distributivity makes it impossible to reduce such terms. 
  
In addition, the ordinal properties of the classical interval arithmetic with respect 
to inclusion ordering ``$\subseteq$'' are unsatisfactory. In partially ordered sets, 
the possibility of taking, for any two elements, their lower bound ``$\wedge$'' and 
the upper bound ``$\vee$'' with respect to the order in question plays a huge role. 
In $\mbb{IR}$, the corresponding operations are 
\begin{align}
\mbf{a}\wedge\mbf{b} \ :=& \ \inf{}_\subseteq \{\,\mbf{a},\mbf{b}\,\}
= \left[\,\max\{\un{\mbf{a}},\un{\mbf{b}}\},\,\min\{\ov{\mbf{a}},
\ov{\mbf{b}}\}\,\right],\hspace*{7mm}\label{Infimum}\\[1mm]
&\makebox[5mm][l]{ --- taking the lower bound %
   with respect to ``$\subseteq$'',}             \notag\\[4mm]
\mbf{a}\vee\mbf{b} \ :=& \ \sup{}_\subseteq \{\,\mbf{a},\mbf{b}\,\}
= \left[\,\min\{\un{\mbf{a}},\un{\mbf{b}}\},\,\max\{\ov{\mbf{a}},
\ov{\mbf{b}}\}\,\right],\hspace*{7mm}\label{Supremum}\\[1mm]
&\makebox[5mm][l]{ --- taking the upper bound %
   with respect to ``$\subseteq$''.}             \notag
\end{align}
Since the first of these operations is not always fulfilled, then $\mbb{IR}$ is, in a sense, 
``not closed'' with respect to the inclusion order.\footnote{If $\mbf{a}, \mbf{b}$ are usual 
one-dimensional intervals with non-empty intersection, then $\mbf{a}\wedge\mbf{b}$ and 
$\mbf{a}\vee\mbf{b}$ coincide with $\mbf{a}\cap\mbf{b}$ and $\mbf{a}\cup\mbf{b}$ respectively. 
But this is not valid in general.}  For example, $[1, 2]\wedge[3, 4]$ is not defined. 
  
The elements of the complete interval arithmetic $\mbb{KR}$ are pairs of real numbers 
$[\,\eta,\,\vartheta\,]$, not necessarily connected by the relation $\,\eta\leq\vartheta$. 
Thus, $\mbb{KR}$ is obtained by appending \emph{improper} intervals $[\,\eta,\vartheta\,]$, 
$\eta > \vartheta$ to the set $\mbb{IR} = \{\,[\,\eta, \vartheta\,] \mid \eta,\,\vartheta 
\in\mbb{R},\;\eta\leq\vartheta\,\}$ of \emph{proper} intervals and real numbers (identified 
with degenerate intervals of zero width). The elements of Kaucher arithmetic and the more 
complex objects formed from them (vectors, matrices) will be highlighted in bold, similar 
to usual intervals. Moreover, if $\mbf{a} = [\,\eta, \,\vartheta\,]$, then $\eta$ is 
called the \emph{left} (or \emph{lower}) end of the interval $\mbf{a}$ and is denoted 
by $\un{\mbf{a}}$ or $\inf\mbf{a}$, and $\vartheta$ is called the \emph{right} (or 
\emph{upper}) endpoint of the interval $\mbf{a}$ and is denoted $\ov{\mbf{a}}$ or 
$\sup{a}$. The interval $\mbf{a}$ is called \emph{balanced} if $\un{\mbf{a}} \ = \ 
-\ov{\mbf{a}}$. 
  
Without trying to replace the rigorous mathematical constructions performed by 
E.\,Kaucher, we will consider the description of the arithmetic $\mbb{KR}$ along 
with informal motivations that help to understand the meaning of various constructions. 
  
Proper and improper intervals, the two halves of $\mbb{KR}$, transform into each 
other as a result of the dualization mapping $\,\dual: \mbb{KR}\to\mbb{KR}$ swapping 
(turning over) the endpoints of the interval, i.\,e. such that 
\begin{equation*}
\label{Dualization}
\dual\mbf{a} := [\;\ov{\mbf{a}},\,\un{\mbf{a}}\;].
\end{equation*}
\emph{Proper projection} of an interval $\mbf{a}$ is the value 
\begin{equation*}
\label{ProperProject}
\pro\mbf{a} := 
\begin{cases}
\ \mbf{a}, &\mbox{ if $\mbf{a}$ is proper, }\\[1mm]
\ \dual\mbf{a} , &\mbox{ otherwise. } 
\end{cases}
\end{equation*}
  
Similar to the classical interval arithmetic $\mbb{IR}$, the ``inclusion'' of one 
interval to another is defined in $\mbb{KR}$ as follows: 
\begin{equation}
\label{Order}
\mbf{a}\subseteq\mbf{b}
\quad\Longleftrightarrow\quad
\un{\mbf{a}}\geq\un{\mbf{b}}
\quad\mbox{ and }\quad
\ov{\mbf{a}}\leq\ov{\mbf{b}}.
\end{equation} 
For example, $[3, 1]\subseteq[2, 2] = 2\in\mbb{R}$. As a consequence, the operations 
of taking the minimum \eqref{Infimum} and the maximum \eqref{Supremum} keep their 
definitions unchanged in $\mbb{KR}$, but now they are always possible due to 
the presence of improper intervals. In particular, $[1, 2]\wedge[3, 4] = [3, 2]$. 
Thus, the extension of $\mbb{IR}$ to $\mbb{KR}$ makes the set of intervals a lattice, 
and even a conditionally complete lattice with respect to the inclusion ordering 
\eqref{Order}.\footnote{A conditionally complete lattice is a partially ordered set 
in which every non-empty bounded subset has exact upper and lower bounds \cite{Birkhoff}. 
This is more than just a lattice, but less than a full lattice, where minima and maxima 
can be taken for any families of elements.} 
  
In addition to the set-theoretic inclusion on the set of intervals $\mbb{KR}$, there 
is another partial ordering, which naturally generalizes the linear order ``$\leq$'' 
on the real axis: 
  
\begin{definition} 
For the intervals $\mbf{a}$, $\mbf{b}\in\mbb{KR}$, we say that \textsl{$\mbf{a}$ 
does not exceed $\mbf{b}$} and write ``$\,\mbf{a}\leq\mbf{b}$'' if and only if 
$\,\un{\mbf{a}}\leq\un{\mbf{b}}\,$ and $\,\ov{\mbf{a}}\leq\ov{\mbf{b}}$. 
\par\noindent 
The interval is called \textsl{nonnegative}, i.\,e. ``$\,\geq 0$'' if both 
its endpoints are nonnegative. The interval is \textsl{nonpositive}, i.\,e. 
``$\,\leq 0$'' if both its endpoints are not positive. 
\end{definition} 
  
For example, $[1, 2]\leq[3, 2]$, and both compared intervals $[1, 2]$ and $[3, 2]$ 
are non-negative. 
  
It is useful to introduce the concept of the \emph{sign of an interval}, which 
we define as 
\begin{equation*}
\sgn\mbf{a} =
\left\{ \
\begin{array}{cl}
+ \ , & \mbox{ if } \;\mbf{a}\geq 0,\\[1mm]
- \ , & \mbox{ if } \;\mbf{a}\leq 0,\\[1mm]
\mbox{undefined}, & \mbox{ if $\,0\,$ is in the interior of $\pro\mbf{a}$}. 
\end{array}
\right.
\end{equation*}
The zero, i.\,e. the zero interval $[0,0]$, may have any sign. 
  
The semigroup of all proper intervals with the operation of addition is fairly simple: 
the addition of intervals is divided into independent operations of addition of the 
left and right endpoints of the operands. As a consequence, the extension of addition 
from $\mbb{IR}$ to $\mbb{KR}$ is easy, and it is defined in $\mbb{KR}$ in exactly 
the same way as in classical interval arithmetic: 
\begin{equation*}
\mbf{a} + \mbf{b} \  := \  \left[\,\un{\mbf{a}}+\un{\mbf{b}},
  \ov{\mbf{a}}+\ov{\mbf{b}}\,\right].
\end{equation*} 
But now it follows from the existence of improper intervals that each element 
$\mbf{a}$ of $\mbb{KR}$ has a unique inverse with respect to addition (also called 
\emph{opposite}), denoted by ``$\opp\mbf{a}$'', and the equality $\,\mbf{a} + 
\opp\mbf{a} = 0\,$ implies that 
\begin{equation} 
\label{Opposite} 
\opp\mbf{a} := [\,-\un{\mbf{a}},-\ov{\mbf{a}}\,]. 
\end{equation} 
With respect to addition, the arithmetic $\mbb{KR}$ is thus a commutative group that 
is isomorphic to the additive group of the standard linear space $\mbb{R}^2$. For 
brevity, we denote by ``$\ominus$'' an operation that is the inverse of addition, 
and it will be called \emph{internal subtraction} in $\mbb{KR}$ (or \emph{algebraic 
subtraction}). Then 
\begin{equation*}
\label{Ominus}
\mbf{a}\ominus\mbf{b}
\ := \ \mbf{a} + \opp\mbf{b}
\ = \ \left[\;\un{\mbf{a}} - \un{\mbf{b}},
 \,\ov{\mbf{a}} - \ov{\mbf{b}}\;\right].
\end{equation*}
  
It is easy to verify that, for the addition in $\mbb{KR}$, the following relation 
is valid 
\begin{equation} 
\label{AddDual} 
\dual\!(\mbf{a} + \mbf{b}) = \dual\mbf{a} + \dual\mbf{b}. 
\end{equation} 
It means, in particular, that the addition of improper intervals in $\mbb{KR}$ is 
a ``mirror image'' of the addition for proper intervals in $\mbb{IR}$. This is consistent 
with the status of improper intervals, so it seems necessary to preserve this property 
when defining other arithmetic operations in $\mbb{KR}$. 
  
We start extending the multiplication to the complete interval arithmetic $\mbb{KR}$ 
with the simplest situation, i.\,e. multiplication by real numbers. It is natural 
to define it in exactly the same way as in classical interval arithmetic: 
\begin{equation} 
\label{ScalarMult} 
\mu\cdot\mbf{a} \ 
:= \ 
\left\{\ 
\begin{array}{ll} 
[\,\mu\,\un{\mbf{a}},\mu\,\ov{\mbf{a}}\,], 
&\mbox{ if }\;\mu\geq 0, \\[1mm] 
[\,\mu\,\ov{\mbf{a}},\mu\,\un{\mbf{a}}\,], 
&\mbox{ otherwise.} 
\end{array}
\right. 
\end{equation} 
  
Further, for nonnegative proper intervals, multiplication in $\mbb{IR}$ is performed 
``through the endpoints'', quite similar to addition. Therefore, it makes sense to define, 
in $\mbb{KR}$, the multiplication of non-negative intervals, which are not necessarily 
proper, according to the same ``endpoint formulas'', i.\,e. as $\mbf{a}\cdot\mbf{b} = 
[\,\un{\mbf{a}}\,\un{\mbf{b}}, \ov{\mbf{a}}\,\ov{\mbf{b}}\,]$. In fact, in this way, 
we embed a multiplicative semigroup of positive intervals into a group. If, in the 
product $\mbf{a}\cdot\mbf{b}$, one of the factors is a nonpositive interval, then we 
can make it non-negative through multiplying by $(-1)$ and using formula 
\eqref{ScalarMult}, thus coming to the previous case. 
  
Note that in all these cases the property analogous to \eqref{AddDual} holds: 
\begin{equation*}
\dual\!(\mbf{a}\cdot\mbf{b}) = \dual\mbf{a}\cdot\dual\mbf{b}. 
\end{equation*}
  
How to extend the definition of multiplication to the entire set $\mbb{KR}$? 
The ability of algebra to do this has already been exhausted, and we need to involve 
considerations concerning the inclusion ordering in $\mbb{KR}$ and the related 
properties of arithmetic operations. 
  
Using the maximum with respect to inclusion \eqref{Supremum}, the fundamental property 
\eqref{FundamProperty}, which defines the operations of classical interval arithmetic, 
can be rewritten in the following equivalent form: 
\begin{align}
\mbf{a}\star\mbf{b} \    \notag 
&= \,\{\,a\star b \mid a\in\mbf{a},\,b\in\mbf{b}\,\} \\[4mm] 
&= \ \raisebox{-0.2ex}{$\Bigl[$}\;\min_{a\in\mbf{a}}\, 
   \min_{b\in\mbf{b}}\, (a\star b), \  \max_{a\in\mbf{a}}\, 
   \max_{b\in\mbf{b}}\, (a\star b) \;\raisebox{-0.2ex}{$\Bigr]$} 
 = \  \bigvee_{a\in\mbf{a}} \ \bigvee_{b\in\mbf{b}} \ (a\star b), 
 \label{FundamPropertyRwr}
\end{align}
where $\star\in\{\,+,-,\cdot\,,/\,\}$. It is easily seen that the addition extended 
to the entire set of proper and improper intervals $\mbb{KR}$, as well as the 
multiplication defined for intervals that do not contain zero and are not contained 
in zero can be represented in a similar way through the operations \eqref{Infimum} 
and \eqref{Supremum} of taking minimum and maximum with respect to inclusion. 
If both operands $\mbf{a}$ and $\mbf{b}$ are improper, then 
\begin{equation} 
\label{DualFundProp} 
\mbf{a}\star\mbf{b} \  = \ 
   \bigwedge_{a\in\pro\mbf{a}} \ \bigwedge_{b\in\pro\mbf{b}} \ (a\star b), 
\end{equation} 
where $\star\in\{+,\cdot\}$. The lower arguments of the operations ``$\wedge$'' 
should have proper projections $\pro\mbf{a}$ and $\pro\mbf{b}$, since improper 
intervals themselves are contained in points due to the definition of inclusion 
in $\mbb{KR}$. 
  
The property \eqref{DualFundProp} is evidently obtained from \eqref{FundamPropertyRwr} 
with the use of dualization. But the next properties of the operation $\star\in
\{\,+,\cdot\,\}$ are not obvious: 
\begin{equation*}
\mbf{a}\star \mbf{b} \  = \ 
   \bigvee_{a\in\mbf{a}} \ \bigwedge_{b\in\pro\mbf{b}} \ (a\star  b), 
\end{equation*}
if $\mbf{a}$ is proper and $\mbf{b}$ is improper, and 
\begin{equation*}
\mbf{a}\star \mbf{b} \  = \ 
   \bigwedge_{a\in\pro\mbf{a}} \ \bigvee_{b\in\mbf{b}} \ (a\star b), 
\end{equation*}
if $\mbf{a}$ is improper and $\mbf{b}$ is proper. Their validity can be verified 
by direct verification. 
  
All the above written formulas, starting with \eqref{FundamPropertyRwr}, can be 
combined into one as follows. We introduce the so-called conditional operation of 
taking the extremum with respect to inclusion: 
\begin{equation}
\label{CondLattOper}
\ii{\mbf{a}}{x} :=
\left\{\
\begin{array}{cl}
\  \displaystyle\bigvee_{x\in\mbf{a}},  
  & \mbox{ if $\mbf{a}$ is proper,}   \\[6mm]
\  \displaystyle\bigwedge_{x\in\dual\mbf{a}}, 
  & \mbox{ if $\mbf{a}$ is improper.}
\end{array}\right. 
\end{equation}
This is an operation that depends on the interval parameter $\mbf{a}$ standing 
as its upper index. The operation is either maximum or minimum with respect to 
the inclusion ``$\subseteq$'', depending on whether $\mbf{a}$ is proper or improper. 
This extremum is taken over all $x$ from the proper projection of the interval 
$\mbf{a}$. Note that any interval $\mbf{a}\in\mbb{KR}$ can be represented as 
\begin{equation*} 
\mbf{a} = 
\ii{\mbf{a}}{x} x 
\end{equation*}
(instead of $x\,$, any letter can be used in the formula). Anyway, for 
$\star\in\{+,\cdot\}$, the following relation is valid: 
\begin{equation} 
\label{LatticeRepres} 
\mbf{a}\star \mbf{b} \  = \ 
\ii{\mbf{a}}{a} \ii{\mbf{b}}{b} \  (a\star b). 
\end{equation} 
This representation, first obtained in \cite{Kaucher-73}, expresses the relationship 
between the result of the interval operation $\mbf{a}\star\mbf{b}\,$ and the results 
of the point operations $a\star b$ for $a\in\pro\mbf{a}$ and $b\in\pro\mbf{b}$. 
It can be taken as a basis for the definition of arithmetic operations in the complete 
interval arithmetic $\mbb{KR}$. 
  
It is not hard to derive, from \eqref{LatticeRepres}, the monotonicity of 
the interval arithmetic operations with respect to inclusion. 
  
In order to write out explicit formulas for multiplication in complete interval 
arithmetic, we select the following subsets in $\mbb{KR}$: 
\begin{align*}
\mcl{P} &:= \{\,\mbf{a}\in\mbb{KR} \mid(\un{\mbf{a}}\geq 0)\;\&\; 
 (\ov{\mbf{a}}\geq 0)\,\}  && \text{ --- nonnegative intervals,}\\[7pt]
\mcl{Z} &:= \{\,\mbf{a}\in\mbb{KR}\mid\un{\mbf{a}}\leq
 0\leq\ov{\mbf{a}}\,\}     && \text{ --- zero-containing intervals,}\\[7pt]
-\mcl{P} &:= \{\,\mbf{a}\in\mbb{KR}\mid-\mbf{a}\in\mcl{P}\,\} 
                           && \text{ --- nonpositive intervals,}\\[7pt]
\dual\mcl{Z} &:= \{\,\mbf{a}\in\mbb{KR} \mid\dual\mbf{a}\in\mcl{Z}\,\} 
                           && \text{ --- intervals contained in zero.}
\end{align*}
Overall, $\mbb{KR} = \mcl{P}\,\cup\,\mcl{Z}\,\cup\,(-\mcl{P})\,\cup\,(\dual\mcl{Z})$. 
Then the multiplication in Kaucher interval arithmetic can be described by 
Table~\ref{KaucherProduct} \cite{Kaucher-80}, the cells of which are obtained 
as a result of detailed writing out the particular cases of applying formula 
\eqref{LatticeRepres} and our previous results. A remarkable fact is that this 
table is the supplement of a similar table for multiplication in classical interval 
arithmetic with one more row and one more column that correspond to the case of 
operands from the set $\dual\mcl{Z}$. 
  
  
\begin{table}[htb] 
\caption{Multiplication in Kaucher complete interval arithmetic} 
\label{KaucherProduct} 
\begin{equation*} 
\def\bx{\mbf{a}} \def\by{\mbf{b}} 
\def\ux{\un{\bx}\,} \def\ovx{\ov{\bx}\,} 
\def\uy{\un{\by}} \def\ovy{\ov{\by}} 
\arraycolsep 5pt 
\renewcommand{\arraystretch}{2.3} 
\begin{array}{|l||cccc|} 
\hline 
\qquad\mbf{$\cdot$} & \by\in\mcl{P} & \by\in\mcl{Z} & \by\in-\mcl{P} & 
\by\in\dual\mcl{Z} \\[1mm]\hline\hline 
\ \bx\in\mcl{P} & \ [\,\ux\uy,\ovx\ovy\,] & [\,\ovx\uy,\ovx\ovy\,] & 
 [\,\ovx\uy,\ux\ovy\,] & [\,\ux\uy,\ux\ovy\,] \\[1mm] 
\ \bx\in\mcl{Z} & \ [\,\ux\ovy,\ovx\ovy\,] & 
 \parbox{35mm}{$[\,\min\,\{\ux\ovy,\ovx\uy\}$,\\*[3pt] %
        $\phantom{M}\max\,\{\ux\uy,\ovx\ovy\}\;]$} & [\,\ovx\uy,\ux\uy\,] & 
 \mbox{\large$0$}\\[1mm]
\ \bx\in-\mcl{P} & \ [\,\ux\ovy,\ovx\uy\,] & [\,\ux\ovy,\ux\uy\,] &
 [\,\ovx\ovy,\ux\uy\,] & [\,\ovx\ovy,\ovx\uy\,] \\[1mm]
\ \bx\in\dual\mcl{Z} & \ [\,\ux\uy,\ovx\uy\,] & \mbox{\large$0$} & 
 [\,\ovx\ovy,\ux\ovy\,] & \parbox{35mm}{$[\,\max\,\{\ux\uy,\ovx\ovy\}$,\\*[3pt]%
     $\phantom{M}\min\,\{\ux\ovy,\ovx\uy\}\;]$} \\[5mm]\hline
\end{array}
\end{equation*}
\end{table}
  
  
As we see, the multiplication in Kaucher arithmetic admits non-trivial zero divisors. 
For example, $[\,-1,\,2\,]\cdot[\,5,\,-3\,] = 0$. The interval multiplication in Kaucher 
arithmetic turns out to be commutative and associative \cite{GardTre-80, Kaucher-77, 
Kaucher-80}, but the multiplication group in $\mbb{KR}$ is formed only by intervals 
$\mbf{a}$ for which $\un{\mbf{a}}\,\ov{\mbf{a}} > 0$ (or, otherwise, $0\not\in\pro 
\mbf{a}$), because no any wider subset of $\mbb{KR}$ satisfies the so-called 
``cancellation law'' 
\begin{equation*} 
\mbf{ab} = \mbf{ac} \quad \Rightarrow \quad \mbf{b} = \mbf{c}. 
\end{equation*}
This is the algebraic condition that a semigroup can be embedded into a group. 
  
Therefore, for any interval $\mbf{a}$ of $\mbb{KR}$ that does not contain zero and 
is not contained in zero itself, there is a single inverse element with respect 
to multiplication, which we will denote by ``$\inv\mbf{a}$''. From the equality 
$\,\mbf{a}\cdot\inv\mbf{a} = 1\,$, it follows that 
\begin{equation}
\label{Inverse}
\inv\mbf{a} := [\,1/\un{\mbf{a}}, 1/\ov{\mbf{a}}\,].
\end{equation}
For brevity, we will denote the inverse operation of the multiplication, the so-called 
internal (algebraic) division in $\mbb{KR}$, by ``$\oslash$'', so that 
\begin{equation*}
\label{Oslash}
\mbf{a}\oslash\mbf{b}
\ := \ \mbf{a}\cdot\inv\mbf{b}
\ = \ \mbf{a}\cdot\left[ \ 1/\un{\mbf{b}}, \,1/\ov{\mbf{b}} \ \right]
\qquad\text{ for } \ 0\not\in\pro\mbf{b}.
\end{equation*} 
  
The above table of explicit formulas for multiplication in the complete interval 
arithmetic is convenient for computer implementation, but rather cumbersome and 
barely foreseeable. In some cases, it makes sense to resort to other formulas 
for interval multiplication in $\mbb{KR}$, which were proposed by A.V.\,Lakeyev 
in \cite{Lakeyev-95, Lakeyev-98}. Recall the following definition (see, for example, 
\cite{Birkhoff,Dieudonne}): 
  
\begin{definition}
\label{RealPosNegParts}
For a real number $a$, the values 
\begin{equation*}
a^+ := \max\{\,a, 0\,\}, \hspace{23mm} a^- := \max\{\,-a, 0\,\} 
\end{equation*}
are called \textsl{positive part} and \textsl{negative part} of $a$ respectively. 
\end{definition}

Then $\,a = a^+ - a^-$ and $|a| = a^+ + a^-$.
  
\begin{proposition} {\normalfont(Lakeyev formulas)} 
\label{Lakeyev} \ 
For any intervals $\,\mbf{a}$, $\mbf{b}\in\mbb{KR}$, there holds 
\begin{multline*} 
\mbf{a}\cdot\mbf{b} \ = \ 
\bigl[ \ 
\max\{\,\un{\mbf{a}}^+\un{\mbf{b}}^+,\,\ov{\mbf{a}}^-\ov{\mbf{b}}^- \} 
- \max\{\,\ov{\mbf{a}}^+\un{\mbf{b}}^-,\,\un{\mbf{a}}^-\ov{\mbf{b}}^+\},\\[1mm] 
\max\{\,\ov{\mbf{a}}^+\ov{\mbf{b}}^+,\,\un{\mbf{a}}^-\un{\mbf{b}}^- \} 
- \max\{\,\un{\mbf{a}}^+\ov{\mbf{b}}^-,\,\ov{\mbf{a}}^-\un{\mbf{b}}^+\} \ 
\bigr].\phantom{(1.25)} 
\end{multline*}
If one of the intervals is $\mbf{a}$, $\mbf{b}$ is proper, then 
\begin{multline} 
\label{PropInteMult} 
\mbf{a}\cdot\mbf{b} \ = \
\bigl[ \ \un{\mbf{a}}^+\un{\mbf{b}}^+ +\ov{\mbf{a}}^-\ov{\mbf{b}}^- 
- \max\{\,\ov{\mbf{a}}^+\un{\mbf{b}}^-,\,\un{\mbf{a}}^-\ov{\mbf{b}}^+\},\\[1mm] 
\ \max\{\,\ov{\mbf{a}}^+\ov{\mbf{b}}^+,\,\un{\mbf{a}}^-\un{\mbf{b}}^- \} 
- \un{\mbf{a}}^+\ov{\mbf{b}}^- -\ov{\mbf{a}}^-\un{\mbf{b}}^+ \ \bigr]. 
\end{multline}
This formula is not simplified if we additionally know that both intervals $\mbf{a}$, 
$\mbf{b}$ are proper.\par\noindent 
If one of the intervals $\mbf{a}$, $\mbf{b}$ is proper and the other is improper, 
then 
\begin{equation}
\label{MixedInteMult}
\mbf{a}\cdot\mbf{b} \ = \
\bigl[ \ \un{\mbf{a}}^{+}\un{\mbf{b}}^{+} +\ov{\mbf{a}}^{-}\ov{\mbf{b}}^{-} - 
  \ov{\mbf{a}}^{+}\un{\mbf{b}}^{-} -\un{\mbf{a}}^{-}\ov{\mbf{b}}^{+}, \ 
\ov{\mbf{a}}^{+}\ov{\mbf{b}}^{+} +\un{\mbf{a}}^{-}\un{\mbf{b}}^{-} - 
  \un{\mbf{a}}^{+}\ov{\mbf{b}}^{-} -\ov{\mbf{a}}^{-}\un{\mbf{b}}^{+} \ \bigr]. 
\end{equation}
\phantom{A}
\end{proposition}
  
The advantage of the Lakeyev formulas is their global character. They give a single and 
uniform expression for the interval product $\mbf{a}\cdot\mbf{b}$ over the entire domain 
of $\mbf{a}$ and $\mbf{b}$, whereas the representation via Table.~\ref{KaucherProduct} 
has a piecewise character. This is inconvenient in the study of interval functions 
``as a whole'', in particular, in the study of differentiability and its analogues, 
in the calculation and evaluation of generalized derivatives, etc. 
  
Subtraction and division in arithmetic $\mbb{KR}$ are defined in the same way 
as in classical interval arithmetic: 
\begin{align*}
\mbf{a} - \mbf{b} \ &:= \
 \mbf{a} + (-1)\cdot\mbf{b} = \left[\,\un{\mbf{a}}
 - \ov{\mbf{b}},\ov{\mbf{a}} - \un{\mbf{b}}\,\right], \\[2mm]
\mbf{a}\,/\,\mbf{b} \ &:= \
 \mbf{a}\cdot\left[\,1/\ov{\mbf{b}},1/\un{\mbf{b}}\,\right]
\qquad\mbox{for} \ 0\not\in\pro\mbf{b}.
\end{align*}

Similar to its classical predecessors, all operations of the complete interval 
arithmetic are \emph{inclusion monotone}, i.\,e. monotone with respect to the partial 
order \eqref{Order}: if $\mbf{a}, \mbf{a}', \mbf{b}, \mbf{b}' \in\mbb{KR}$, then
\begin{equation*}
\mbf{a}\subseteq\mbf{a}',\ \mbf{b}\subseteq\mbf{b}'
\quad\Rightarrow\quad\mbf{a}\star\mbf{b}\subseteq\mbf{a}'\star\mbf{b}'
\end{equation*}
for any arithmetic operation $\star\in\{\;+,\,-,\,\cdot\,,\,/\;\}$. This follows from 
their definition according to formula \eqref{LatticeRepres}. 
  
The relationship of addition and multiplication in Kaucher arithmetic is expressed 
by the following inclusions: 
\begin{align}
\text{ if $\mbf{a}$ is proper}, & \text{ then } 
  \ \mbf{a}\cdot(\mbf{b}+\mbf{c}) \,\subseteq\,
  \mbf{a}\cdot\mbf{b} + \mbf{a}\cdot\mbf{c}\label{Subdistr}\\[1mm]
& \text{ --- subdistributivity,}\nonumber \\[2mm]
\text{ if $\mbf{a}$ is improper}, & \text{ then } 
  \ \mbf{a}\cdot(\mbf{b}+\mbf{c}) \,\supseteq\,
  \mbf{a}\cdot\mbf{b} + \mbf{a}\cdot\mbf{c}\label{Superdistr}\\[1mm]
& \text{ --- superdistributivity.}\nonumber
\end{align}
These inclusions turn to exact equalities when, in particular, $\mbf{a}$ squeezes 
to a point, that is, $\mbf{a} = a\in\mbb{R}$:
\begin{equation}
\label{KRPointDistr}
a\cdot(\mbf{b}+\mbf{c})\,= \,a\cdot\mbf{b} + a\cdot\mbf{c},
\end{equation}
Another important case of distributivity is the case when the signs of the intervals 
$\mbf{b}$ and $\mbf{c}$ coincide with each other: 
\begin{equation}
\label{KRSignDistr}
\mbf{a}\cdot(\mbf{b}+\mbf{c})\,= \,\mbf{a}\cdot\mbf{b} + \mbf{a}\cdot\mbf{c},
\qquad\mbox{ if } \ \mbf{b\,c}\geq 0.
\end{equation} 
  
E.\,Garde\~{n}es et al. introduced in \cite{GardTre-80}, for a complete description 
of all cases of distributivity, the concept of \emph {distributive areas} defined 
by an interval $\mbf{a}$. Membership of the operands in these distributivity areas 
leads to equalities instead of inclusions \eqref{Subdistr}--\eqref{Superdistr}. 
Later, S.\,Markov and his co-workers, as a classification of various particular 
cases of the distributivity of addition with respect to multiplication in $\mbb{KR}$, 
proposed a ``generalized distribution law'' \cite{DimiMarkPop,MarPopUll,Popova}, 
covering a large number of various situations. Of the variety of cases considered 
in these articles, we will further need the following relation: 
\begin{equation}
\label{MarkovDistr}
\mbf{a}\cdot(\mbf{b}+\mbf{c}) =
\mbf{a}\cdot\mbf{b} + (\dual\mbf{a})\cdot\mbf{c},
\end{equation}
if the intervals $\,\mbf{b}$, $\mbf{c}$ and $\mbf{b}+\mbf{c}\,$ have definite signs 
and $\,\sgn\mbf{b} = -\sgn\mbf{c} = \sgn(\mbf{b}+\mbf{c})$.

  
\section{Theory}

In this section, we consider the basic theoretical facts concerning interval linear 
systems of equations \cite{Shary-95,Shary-96,Shary-98,SharySurvey}. 
  
Despite the simple structure of the system of equations \eqref{InLinSys1}, we can 
use, for its solution, some elimination methods, symbolic transformations, etc., only 
in very particular situations. The reason is insufficient algebraic properties of the 
interval arithmetics. The absence of full distributivity in Kaucher arithmetic makes 
it generally impossible to perform even such a simple operation as reduction of similar 
terms. It is for this reason that the methods considered in our work are essentially 
\emph{numerical}. An important theoretical result on formal solutions to interval 
linear systems was obtained by A.V.\,Lakeyev who managed to show the NP-complexity 
(intractability) of computing formal (algebraic) solutions to interval systems 
of linear equations in general form \cite{Lakeyev-95, Lakeyev-96}. 
  
To find formal solutions of interval linear systems of equations, several numerical 
methods were proposed, of which the subdifferential Newton method is the most efficient 
(see \cite{Shary-95,Shary-98,SharyBook}, and its computer implementations are freely 
available at \cite{SharyCodes}). The goal of this paper is the development of stationary 
single-step iterative methods for computing formal (algebraic) solutions to interval 
systems of linear algebraic equations. The need to build such methods is due 
to a number of facts. Despite very high efficacy of the subdifferential Newton method 
in practice, its justification for the most general case faces a number of difficulties. 
In addition to calculating the solution, the methods of the type we are going to construct 
provide also a proof of the uniqueness of the found solution, which is not provided 
by the subdifferential Newton method. Finally, another reason for the need to develop 
single-step stationary iterative methods is the fact that they are able to solve interval 
systems of equations that are not linear in form, for example, 
\begin{equation*} 
\mbf{A}x = \mbf{b}(x), 
\end{equation*} 
where $\mbf{b}(x)$ is an interval function of the unknown variable $x$. 
  
To date, computational mathematics has accumulated a large arsenal of theoretical 
approaches and efficient practical algorithms for solving a wide variety of equations 
and systems of equations. Can we use any of these methods? Is it possible to apply to 
our problem any of the traditional numerical methods for solving ``operator equations''? 
Yes, but with some reservations and modifications. 
  
The majority of traditional methods for the solution of equations and systems of equations 
relate to operator equations in linear spaces. Formally, these methods are not applicable 
to the problem of computing the formal (algebraic) solutions of system \eqref{InLinSys1}, 
since $\mbb{IR}^n$ and $\mbb{KR}^n$ are not linear spaces (see \cite{Shary-98}). However, 
we can easily circumvent this difficulty by embedding the space $\mbb{IR}^n$ or $\mbb{KR}^n$ 
into the usual well-studied Euclidean space $\mbb{R}^{2n}$. 
  
As we have already noted, the problem of finding formal solutions to interval equations 
is, in essence, the traditional mathematical problem of solving some equations, and most 
of the classical numerical analysis is devoted to solving such problems. But the peculiarity 
of our situation is that the basic set $\mbb{KR}^n$, on which the equations to be solved 
are considered, is not a linear space at all: the lack of distributivity in interval 
arithmetic leads to a violation of the axiom of linear space that requires the fulfillment 
of the identity 
\begin{equation*} 
(\mu + \nu) \,\mbf{x} = \mu\, \mbf{x} + \nu\,\mbf{x} 
\end{equation*}
for all $\mbf{x}\in\mbb{KR}^n$ and any scalars $\mu, \nu\in\mbb{R}$. Thus, most of the existing 
approaches to the study of operator equations and to the calculation of their solutions are not 
directly applicable to our problem. 
  
Moreover, remaining within the interval space $\mbb{KR}^n$, we will not be able to perform 
a theoretical analysis of the situation and understand some phenomena. For example, 
the point matrix 
\begin{equation} 
\label{BadMatrix}
\left(
\begin{array}{rr}
 1 & 1 \\ 
-1 & 1
\end{array} 
\right) 
\end{equation}
is regular (non-singular) in the sense of classical linear algebra, but multiplication 
by this matrix in $\mbb{KR}^2$ can nullify a non-zero vector: 
\begin{equation*} 
\left( 
\begin{array}{rr}
 1 & 1 \\[1mm]
-1 & 1 
\end{array} 
\right) 
\cdot
\left( 
\begin{array}{c} 
[-1, 1] \\[1mm] 
[1, -1]
\end{array}
\right) =
\left(
\begin{array}{c}
0 \\[1mm] 0
\end{array}
\right).
\end{equation*}
What is the reason? It is hardly possible to detect it from within the interval space, 
which is essentially non-linear. So, there is an urgent need to transfer our considerations 
to a certain \emph{linear space}, which we denote by $ U $ for generality. We also assume 
that a topological structure is determined on $U$ consistent with its linear structure. 

From an abstract mathematical point of view, we have two different spaces, the interval 
space $\mbb{KR}^n$ and the linear space $U$, on which essentially different algebraic 
structures are given. How is it possible to ``jump over'' from the first one to the second? 
We are going to do it in a way similar to a change of variables, which is defined 
in the following subsection and is called \emph{immersion}. 
  
  
\subsection{Definition and main properties}

To transfer our considerations from the interval space to a linear one, we should build 
some mapping 
\begin{equation*}
\iota: \mbb{KR}^n \to U,
\end{equation*}
--- \emph{embedding} of the interval space $\mbb{KR}^n$ into the linear space $U$. 
It must be \emph{bijective} (one-to-one mapping ``to'') in order to correctly restore 
the interval preimage by its image in $U$, and vice versa. Further, it is easy 
to understand that any bijection $\iota: \mbb{KR}^n \to U$ also generates a~bijection 
from the set of all mappings of $\mbb{KR}^n$ into itself onto the set of all mappings 
of $U$ to itself. More precisely, each $\varphi: \mbb{KR}^n \to\mbb{KR}^n$ is 
associated with a uniquely determined mapping 
\begin{equation}
\label{InducedMapping}
\iota\circ\varphi\circ\iota^{-1}: U\to U,
\end{equation}
where ``$\circ$'' denotes a composition of mappings. Thus, we can argue consider 
mappings of linear spaces as ``exact copies'' of interval mappings. 

\begin{definition}
For an interval mapping $\varphi: \mbb{KR}^n \to\mbb{KR}^n$ and a fixed embedding 
$\;\iota: \mbb{KR}^n \to U$, we will refer to the mapping of the linear space $U$ into 
itself which is determined by \eqref{InducedMapping} as \textsl{induced mapping} for 
$\varphi$ (or, expanded, \textsl{$\iota$-induced}). 
\end{definition}
  
Visually, the situation is represented by a commutative diagram in Fig.~\ref{ComDiagram}. 
  
  
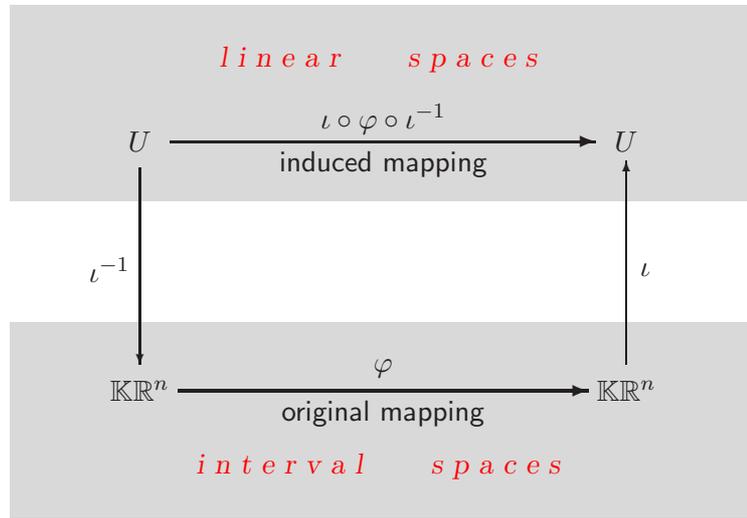
\begin{figure}[htb]
\small\setlength{\unitlength}{1mm}
\hfill\begin{picture}(72,75)
\put(37,59){\makebox(0,0){\fcolorbox[gray]{1}{0.85}{%
\makebox(96,24)[t]{\rule{0mm}{7mm}\sl\color{red}
  l i n e a r \qquad s p a c e s }}}}
\put(37,17){\makebox(0,0){\fcolorbox[gray]{1}{0.85}{%
\makebox(96,24)[b]{\rule[-5mm]{0mm}{0mm}\sl\color{red} 
  i n t e r v a l \qquad s p a c e s }}}}
{\thicklines \put(10,21){\vector(1,0){54}} \put(9,54){\vector(1,0){56}}}
\put(5,50.5){\vector(0,-1){26}} \put(69,24.5){\vector(0,1){27}}
\put(5,21){\makebox(0,0){$\mbb{KR}^n$}}
\put(69,21){\makebox(0,0){$\mbb{KR}^n$}}
\put(37,24){\makebox(0,0){$\varphi$}}
\put(1,37){\makebox(0,0){$\iota^{-1}$}}
\put(71.5,37){\makebox(0,0){$\iota$}}
\put(37,57){\makebox(0,0){$\iota\circ\varphi\circ\iota^{-1}$}}
\put(5,54){\makebox(0,0){$U$}} \put(69,54){\makebox(0,0){$U$}}
\put(37,18){\makebox(0,0){\sf original mapping}}
\put(37,51){\makebox(0,0){\sf induced mapping}}
\end{picture}\hfill\null
\caption{How an immersion $\iota$ generates an induced mapping.} 
\label{ComDiagram}
\end{figure}
  
  
The properties of the mappings $\varphi$ and $(\,\iota\circ\varphi\circ\iota^{-1})$ are 
closely related, so instead of the study of $\varphi$, one can investigate the induced 
mapping $(\,\iota\circ\varphi\circ\iota^{-1})$. Moreover, we can replace the problem 
of solving the equation in $\mbb{KR}^n$ with the problem of solving the equation 
in the linear space $U$, coming to a situation more familiar to modern numerical 
analysis.

\begin{definition} 
Let the equation be given in the interval space $\mbb{KR}^n$, 
\begin{equation}
\label{IntervalEquation}
\varphi (x) = \psi(x),
\end{equation}
where $\varphi, \psi: \mbb{KR}^n\to\mbb{KR}^n$ are some mappings, and an embedding 
$\iota:\mbb{KR}^n\to U$ is fixed. We shall call the \textsl{induced equation} for 
\eqref{IntervalEquation} such an equation 
\begin{equation*} 
\Phi (y) = \Psi(y)
\end{equation*} 
in the linear space $U$, that $\,\Phi$ and $\Psi$ are induced mappings for $\varphi$ 
and $\psi$ respectively, i.\,e. $\Phi = \iota\circ\varphi\circ\iota^{-1}$ and $\Psi 
= \iota\circ\psi\circ\iota^{-1}$.
\end{definition}

Thus, the initial interval equation
\begin{equation}
\label{IniEquation}
\varphi (x) = \psi(x)
\end{equation}
has a formal solution $\mbf{x}^\ast\in\mbb{KR}^n$ if and only if the induced equation
\begin{equation*} 
\Phi (y) = \Psi(y)
\end{equation*} 
has a solution $y^\ast\in U$ in the linear space. In this case, the desired formal 
interval solution $\mbf{x}^\ast$ for \eqref{IniEquation} is uniquely reconstructed 
by $y^\ast$ from the relation 
\begin{equation*} 
\mbf{x}^\ast = \iota^{-1} (\, y^\ast).
\end{equation*} 
  
We are interested in the specific situation with interval linear equations 
\eqref{InLinSys1}--\eqref{InLinSys2}. We can change the original problem 
--- finding solution of the equation 
\begin{equation*} 
\mbf{f}(x) = 0,
\end{equation*} 
such that 
\begin{equation*}
\mbf{f}: \mbf{x}\mapsto\mbf{Ax}\ominus\mbf{b} 
\end{equation*}
to the problem of solving equations 
\begin{equation*}
\eus{F}(y) = \iota(0) 
\end{equation*}
in the linear space $U$ with induced mappings 
\begin{equation*}
\eus{F} = \iota\circ\mbf{f}\circ\iota^{-1}: U\to U, 
\end{equation*}
defined as
\begin{equation*} 
\eus{F} (y) =
\iota\bigl (\,\mbf{A}\iota^{-1}(y)\ominus\mbf{b}\,\bigr). 
\end{equation*} 
  
A more general consideration. Since $\iota$ and $\iota^{-1}$ are bijections, then 
the invertibility of any mapping $\varphi$ on the interval space is equivalent to the 
invertibility of the $\iota$-induced map $\Phi := \iota \circ\varphi\circ\iota^{-1}$ 
acting on the linear space $U$. Herewith 
\begin{equation}
\label{InvMap}
\varphi^{-1} = \iota^{-1}\circ\Phi^{-1}\circ\iota. 
\end{equation}
  
The main question concerning the construction of the embedding of the interval space 
into the linear space is to choose a reasonable compromise between the simplicity of 
this mapping and the convenient form of induced mappings \eqref{InducedMapping}. Among 
all bijective embedding of $\iota:\mbb{KR}^n\to U$, it makes sense to select special 
embeddings that 
\begin{itemize}
\item[1)]
preserve the additive algebraic structure of $\mbb{KR}^n$, \\*  
i.\,e. such that $\iota(\mbf{u} + \mbf{v}) = \iota(\mbf{u}) + \iota(\mbf{v})$ 
for any $\mbf{u}, \mbf{v}\in\mbb{KR}^n$, 
\item[2)]
preserve the topological structure of $\mbb{KR}^n$, \\* 
i.\,e. such that both the mapping $\iota:\mbb{KR}^n\to U$ itself \\* 
and its inverse $\iota^{-1}:U\to\mbb{KR}^n$ are continuous. 
\end{itemize}
Embeddings $\mbb{KR}^n\to U$ satisfying two above prescribed conditions will be called 
\emph{immersions} of the interval space $\mbb{KR}^n$ into the linear space $U$. Thus, 
formally we accept the following

\begin{definition} \textnormal{\cite{Shary-95}} 
\setcounter{ImmersDefi}{\value{defi}} 
Let $U$ be a linear space. A bijective mapping $\iota:\mbb{KR}^n\to U$ will be called 
\textsl{immersion} of $\mbb{KR}^n$ in $U$, if it satisfies the following properties:\\[1mm] 
\hspace*{1em}{\normalfont(1)} $\iota$ is an isomorphism of additive
  groups $\mbb{KR}^n$ and $U$, \\
\hspace*{1em}{\normalfont(2)} $\iota$ is a homeomorphism
  of topological spaces $\mbb{KR}^n$ and $U$.
\end{definition}

For example, if the interval $\mbf{v}\in\mbb{KR}$ is matched with a pair of numbers 
$(\un{\mbf{v}}, \ov{\mbf{v}})\in\mbb{R}^2$, i.\,e. its endpoints, ``forgetting'' about 
their interval sense, then the mapping $\mbb{KR} \to\mbb{R}^2$ is an immersion. This 
example is typical in some sense, since, by involving dimension considerations (see, 
e.\,g., \cite{Engelking}), it is easy to show that Definition~\arabic{ImmersDefi} 
determines the linear space $U$ uniquely: \emph{$U$ must be the Euclidean space 
$\mbb{R}^{2n}$}. This fact is in good agreement with our analytical intuition, and 
we do not give here its strict substantiation, so as not to overload the overgrown 
text of the article. The purpose of this preparatory section is the study of the 
simplest properties of immersions that we will need in the future, when solving 
the induced equations. 
  
Denote by $0_{\mbb{KR}^n}$ and $0_{\mbb{R}^{2n}}$ zero vectors in the spaces $\mbb{KR}^n$ 
and $\mbb{R}^{2n}$ respectively. It immediately follows from Definition~\arabic{ImmersDefi}, 
that for any immersion $\iota:\mbb{KR}^n\to\mbb{R}^{2n}$, we have 
\begin{align} 
\iota (\; 0_{\mbb{KR}^n}) \ &= \ 0_{\mbb{R}^{2n}}, \nonumber  \\[2mm]
\iota(\opp\mbf{x}) \ & = \ -\iota(\mbf{x}),\qquad\mbf{x}\in\mbb{KR}^n.
\label{Opposite1}
\end{align}
At the same time, 
\begin{equation*}
\iota(\mbf{x})\neq 0 \ \mbox{ in }\mbb{R}^{2n}
\quad\Longleftrightarrow\quad
\mbf{x}\neq 0 \ \mbox{ in }\mbb{KR}^n.
\end{equation*}
In addition, the inverse of the immersion map $\iota^{-1}: \mbb{R}^{2n}\to\mbb{KR}^n$ also 
satisfies conditions similar to (1)--(2) from Definition~\arabic{ImmersDefi}, and 
\begin{align}
\iota^{-1}(\;0_{\mbb{R}^{2n}} ) \ &= \ 0_{\mbb{KR}^n}, \nonumber \\[2mm]
\iota^{-1} (- x) \ &= \ \opp\iota^{-1} (x),\qquad x\in\mbb{R}^{2n}.
\label{Opposite2}
\end{align}

\begin{proposition}
\label{HomogePropo}
The immersion is a positive-homogeneous map, i.\,e.
\begin{equation*}
\iota (\lambda\mbf{x}) = \lambda\,\iota(\mbf{x})\quad
\text{ for all } \ \mbf{x}\in\mbb{KR}^n \text{ and } \lambda \geq 0.
\end{equation*}
The mapping $\mbb{R}^{2n}\to\mbb{KR}^n$, inverse to an immersion, is also positively 
homogeneous. 
\end{proposition}

\noindent
\textbf{Proof} is standard. Let $\mbf{x}\in\mbb{KR}^n$. If $\lambda = k$ is a natural 
number, then 
\begin{equation*} 
\iota(k\mbf{x}) \ =  \ 
\iota (\,\underbrace{\mbf{x} + \mbf{x} + \cdots + \mbf{x}}_k\,) 
\ = \ k\,\iota (\mbf{x}). 
\end{equation*}
If $\lambda = 1/l$ for some natural $l$, then, from 
\begin{equation*} 
l\,\iota(\lambda\mbf{x})\; = \; 
\underbrace{\iota(\lambda\mbf{x}) + \iota (\lambda\mbf{x}) + \cdots 
 + \iota (\lambda\mbf{x})}_l \ = \  \iota(l\lambda x)\; = \; \iota(\mbf{x}), 
\end{equation*} 
it follows that 
\begin{equation*}
\iota(\lambda\mbf{x}) \ = \ l^{-1}\iota(\mbf{x})
\ = \ \lambda\,\iota (\mbf{x}).
\end{equation*}
    
If $\lambda = k/l$ for natural numbers $k$ and $l$, then, using the already considered 
cases, we obtain 
\begin{equation*}
\iota(\lambda\mbf{x}) \ = \iota\left(\frac{k}{l}\,\mbf{x}\right)
\ = \ k\,\iota\left(\frac{1}{l}\,\mbf{x}\right) \ = \frac{k}{l}\,
\iota (\mbf{x}) \ = \ \lambda\,\iota (\mbf{x}).
\end{equation*}
Hence, the equality $\,\iota(\lambda\mbf{x}) = \lambda\,\iota (\mbf{x})\,$ is valid 
for all nonnegative rational numbers $\lambda$. Extending it to all non-negative real 
numbers can be done by passing to the limit, using the continuity of the immersion 
$\iota$. 
  
The proof for a mapping $\iota^{-1}$ which is inverse to an immersion is performed 
in a completely similar way. 
\hfill$\blacksquare$

\begin{proposition}
\label{ImmersPropo}
If $\iota: \mbb{KR}^n\to\mbb{R}^{2n}$ is an immersion, and $T$ is a non-singular linear 
transformation of the space $\mbb{R}^{2n}$, then $(T\circ\iota)$ is also an immersion. 
\par\noindent 
Conversely, any other immersion $\varkappa :\mbb{KR}^n\to\mbb{R}^{2n}$ can be represented 
as $(T\circ\iota)$ for some nonsingular linear transformation $T :\mbb{R}^{2n}\rightarrow 
\mbb{R}^{2n}$. 
\end{proposition}

\begin{proof}
The first part of Proposition is justified trivially. 
  
To prove the second part, we consider the mapping $(\,\varkappa\circ\iota^{-1})$. Being 
a composition of two isomorphisms, it is an automorphism of the additive group of the linear 
space $\mbb{R}^{2n}$, and by virtue of Proposition~\ref{HomogePropo} this map is also 
positively homogeneous. Also, for any $x\in\mbb{R}^{2n}$ 
\begin{equation*}
(\, \varkappa\circ\iota^{-1})(x - x) \ = \ 0_{\mbb{R}^{2n}}\,
= \ (\, \varkappa\circ\iota^{-1})(x) + (\, \varkappa\circ\iota^{-1}) (- x), 
\end{equation*}
which implies 
\begin{equation*}
(\, \varkappa\circ\iota^{-1}) (- x) = - (\, \varkappa\circ\iota^{-1}) (x). 
\end{equation*}
Hence, we get the homogeneity of $(\, \varkappa\circ\iota^{-1})$ with respect to multiplication 
by negative numbers as well. 
  
Thus, in general, the mapping $\varkappa\circ\iota^{-1}\,$ turns out to be a nonsingular 
linear transformation of the space $\mbb{R}^{2n}$. We can therefore take $\,T = \varkappa
\circ \iota^{-1}$. 
\end{proof}
  
  
\subsection{Standard immersion} 
\label{StaImmersionSubse}

From Proposition~\ref{ImmersPropo}, it follows that any two immersions of $\mbb{KR}^n$ into 
$\mbb{R}^{2n}$ that satisfy Definition~\arabic{ImmersDefi}, are the same to within a linear 
transformation of $\mbb{R}^{2n}$. In fact, the choice of a convenient immersion turns out 
to be even more constrained, since in interval analysis, in addition to the algebraic 
properties taken into account by Definition~\arabic{ImmersDefi}, the structure of the partial 
order with respect to inclusion plays an important role. It should also be ``adequately'' 
transferred by immersion from $\mbb{KR}^n$ to $\mbb{R}^{2n}$. 
  
Every immersion $\iota :\mbb{KR}^n\to\mbb{R}^{2n}$ naturally generates a partial ordering 
``$\sqsubseteq$'' in the linear space $\mbb{R}^{2n}$, which is the image of the inclusion 
ordering ``$\subseteq$'' in $\mbb{KR}^n$ under the immersion $\iota$. Namely, 
\begin{equation}
\label{OrderDefi}
\begin{array}{c}
x\sqsubseteq y, 
  \mbox{ i.\,e., ``$x$ does not exceed $y$'' \ in $\mbb{R}^{2n}$ }\\[5pt]
\Updownarrow\\[7pt]
\iota^{-1}(x)\,\subseteq\,\iota^{-1}(y)\ \mbox{ in } \ \mbb{KR}^n.
\end{array}
\end{equation}
  
\begin{definition}
The partial ordering ``$\,\sqsubseteq$'' on $\mbb{R}^{2n}$, defined through \eqref{OrderDefi}, 
will be called \textsl{induced partial ordering}. 
\end{definition}
  
For any $x,\,y,\,u,\,v\in\mbb{R}^{2n}$, there holds 
\begin{eqnarray*}
&& x\sqsubseteq y,\ \alpha\geq 0 \quad\Rightarrow\quad
\alpha x\sqsubseteq\alpha y,\\[2mm]
&& x\sqsubseteq y,\ u\sqsubseteq v \quad\Rightarrow\quad
x+u\sqsubseteq y+v,
\end{eqnarray*} 
and in such cases it is said that the partial order ``$\sqsubseteq$'' is consistent with 
the linear structure on $\mbb{R}^{2n}$ \cite{KantoAki,Schaefer}. As a consequence, this partial 
order can be equivalently defined by specifying the so-called cone of positive elements, 
i.\,e. the set $K_\sqsubseteq = \{\,x\in\mbb{R}^{2n}\mid 0\sqsubseteq x\,\}$ \cite{Krasnosel, 
Schaefer}. 
  
Recall that a \emph{cone} in a linear topological space is a closed convex positively 
invariant set that does not contain any one-dimensional subspace. As is known, in a partially 
ordered linear space, where the order is consistent with the linear structure, the set of 
positive elements is a cone. Conversely, the assignment of the cone $K_\sqsubseteq$ 
uniquely determines the partial ordering of the space, for which 
\begin{equation*}
x\sqsubseteq y
  \quad\Longleftrightarrow\quad
  y - x\in K_\sqsubseteq .
\end{equation*}
  
It is clear that the specific formulas determining the induced order of ``$\sqsubseteq$'' 
depend on the form of the immersion $\iota$. But on the Euclidean space $\mbb{R}^{2n}$, 
the simplest and most convenient is to set the order in a component-wise manner, when 
\begin{equation}
\label{StandardOrder}
x\leq y\quad\Longleftrightarrow\quad
x_i\leq y_i, \ i = 1,2,\ldots,2n.
\end{equation}
Accordingly, the cone of positive elements with this ordering $\mbb{R}^{2n}$ is the set 
\begin{equation*}
K_\leq = \{\,x\in\mbb{R}^{2n}\mid \,x_i\geq 0,\,i = 1,2,\ldots,2n\,\}, 
\end{equation*} 
i.\,e., the positive orthant of the space $\mbb{R}^{2n}$. It is therefore natural to require 
from the immersion that the order \eqref{OrderDefi} induced by it coincides with this simple 
component-wise order \eqref{StandardOrder}, i.\,e. that 
\begin{equation}
\label{InducedOrder}
x\sqsubseteq y\quad\Longleftrightarrow\quad
x\leq y \ \mbox{ in component-wise sense. }
\end{equation}
For what immersion $\mbb{KR}^n\to\mbb{R}^{2n}$, is this possible? 
  
It is easy to see that the required immersion is the so-called \emph{standard immersion}, 
first introduced in \cite{Shary-95}: 
  
\begin{definition}
The immersion $\,\sti:\mbb{KR}^n \to\mbb{R}^{2n}$ that acts according 
to the rule 
\begin{equation}
\label{StandardImmersion}
(\,\mbf{x}_1, \mbf{x}_2, \ldots, \mbf{x}_n)
\mapsto
(\,-\un{\mbf{x}}_1, -\un{\mbf{x}}_2, \ldots , -\un{\mbf{x}}_n ,
\ov{\mbf{x}}_1, \ov{\mbf{x}}_2, \ldots , \ov{\mbf{x}}_n),
\end{equation}
i.\,e., such that the negated left endpoints of the intervals $\mbf{x}_1$, $\mbf{x}_2$, 
\ldots, $\mbf{x}_n$ become the first, second, \ldots, $n$-th component of the point 
$2n$-vector, while the right endpoints of the intervals $\mbf{x}_1, \mbf{x}_2, \ldots, 
\mbf{x}_n$ become the $(n+1)$-st, \ldots, $2n$-th components of the point $2n$-vector 
respectively will be called \textsl{standard immersion} of the interval space $\mbb{KR}^n$ 
into $\mbb{R}^{2n}$. 
\end{definition}

\begin{corollary}. 
From definition \eqref{OrderDefi} of the induced order on $\mbb{R}^{2n}$ and 
the requirement \eqref{InducedOrder} for the standard immersion sti, it is easy 
to deduce that 
\begin{equation}
\label{SigmaProperty1}
\sti\biggl(\;\bigvee_{\vartheta\in\varTheta}\;\mbf{x}_\vartheta\;\biggr) 
\ = \  \sti\,\Bigl(\;\sup\limits_{\vartheta\in\varTheta}{}_\subseteq \ 
\mbf{x}_\vartheta\Bigr) \ = \ \sup\limits_{\vartheta\in\varTheta}{}_\leq  
\  \sti\,(\,\mbf{x}_\vartheta)
\end{equation}
for any bounded family of interval vectors $\{\,\mbf{x}_{\vartheta}\in\mbb{KR}^n 
\mid\vartheta\in\varTheta\,\}$, where $\varTheta$ is an index set. Therefore, the standard 
immersion maps supremums with respect to inclusion in the interval space $\mbb{KR}^n$ 
to supremums with the component-wise order on $\mbb{R}^{2n}$. The same is true for 
infimums. 
\end{corollary} 
  
So, the fact that the induced partial order on the linear space $\mbb{R}^{2n}$ coincides 
with the usual component-wise ordering and, hence, the simplification of calculations 
and reasoning are the main justification for the form \eqref{StandardImmersion} chosen 
for immersion called ``standard''. Moreover, the foregoing rather strongly suggests that 
in the theoretical part of our work we should consider only the standard immersion of 
the form \eqref{StandardImmersion}, although other immersions can sometimes be practically 
useful. For example, in the computer implementation of the algorithms described in this 
book, the author often used the simplest immersion 
\begin{equation*}
(\,\mbf{x}_1,\mbf{x}_2, \ldots, \mbf{x}_n) \mapsto 
   (\un{\mbf{x}}_1, \un{\mbf{x}}_2, \ldots, \un{\mbf{x}}_n, 
   \ov{\mbf{x}}_1, \ov{\mbf{x}}_2, \ldots, \ov{\mbf{x}}_n) 
\end{equation*}
which is more convenient for practical programming, etc. 
  
It is useful to provide a methodological commentary on the content of this and 
the preceding subsections. The method of identifying the endpoints of an interval 
or interval vector with the components of a point vector in the Euclidean space 
of double dimension was often used by researchers. But we singled out the procedure 
of this identification into a separate notion --- immersion $\mbb{KR}^n\to\mbb{R}^{2n}$ 
--- and undertook its thorough investigation. For what purpose? Could it be possible 
to do without ``unnecessary abstractions''?
  
In addition to the fact that explicit and conscious operating with any object is always 
more preferable than the implicit one, ``by default'', there are at least two more 
reasons to consider the immersion as an independent concept: 
\begin{itemize}
\parskip 0mm
\item[\color{blue}$\bullet$] 
the mapping $\mbb{KR}^n \to\mbb{R}^{2n}$ cannot be determined in a unique way 
once and for all, which would be the most convenient (natural, etc.) for all 
possible practical situations; 
\item[\color{blue}$\bullet$] 
we can get a tangible benefit from this non-uniqueness, i.\,e. the most complete 
use of the features of immersions in each specific case. 
\end{itemize}
It is easy to see that both of these arguments really apply in our situation.

  
\subsection{Extended multiplier matrix}

\begin{theorem}
\setcounter{MatrixTheo}{\value{theo}} 
Let the mapping $\,\iota :\mbb{KR}^n \to\mbb{R}^{2n}$ be an immersion and 
\begin{equation*}
\phi :\mbb{KR}^n\to\mbb{KR}^n
\end{equation*} 
is operator of multiplication by a point square matrix in the space $\mbb{KR}^n$, i.\,e. 
\begin{equation*}
\phi(\mbf{x})=Q\mbf{x}
\end{equation*}
for some $Q\in\mbb{R}^{n\times n}$, $Q=(\,q_{ij})$. Then the induced mapping 
$\;(\iota\circ\phi\circ\iota^{-1})$ is a linear transformation of the space $\mbb{R}^{2n}$.
For the standard immersion ``{\,\rm sti}'', the matrix of this induced linear transfor\-mation 
$(\,\sti\circ\phi\circ\sti^{-1})$ is a block $2n\times 2n$-matrix of the form 
\begin{equation}
\label{BlockForm}
\arraycolsep 3mm
\renewcommand{\arraystretch}{1.5}
\left( \begin{array}{c|c}
Q^+ & Q^- \\ \hline
Q^- & Q^+
\end{array} \right),
\end{equation}
where $n\times n$-submatrices $Q^+ = (\,q_{ij}^+)$ and $Q^- = (\,q_{ij}^-)$ are 
positive and negative parts of $Q$, i.\,e. the matrices made up of positive and 
negative parts of the elements of $Q$ respectively. 
\end{theorem}
  
Previously, such matrices were called ``concomitant matrices'' (see e.\,g. 
\cite{SharySurvey}, but then it turned out that this term was already used 
in mathematics and has a different meaning. 
  
\bigskip 
\begin{proof}
In order to substantiate the first statement of the Theorem, we need to show 
the \emph{additivity} and \emph{homogeneity} of the mapping 
\begin{equation*}
\iota\circ\phi\circ\iota^{-1}: \mbb{R}^{2n}\to\mbb{R}^{2n}.
\end{equation*}
  
The additivity of $\phi$ immediately follows from the distributivity relation 
\begin{equation*} 
q\cdot (\mbf{x} + \mbf{y}) = q\cdot\mbf{x} + q\cdot\mbf{y}, 
\end{equation*} 
which is valid for any intervals $\mbf{x}$, $\mbf{y}\in\mbb{KR}$ provided that $q$ is 
a point (non-interval) value. The immersion $\iota$ and its inverse mapping $\iota^{-1}$ 
are also additive. Consequently, the induced mapping $\;\iota\circ\phi\circ\iota^{-1}$ 
is additive as a composition of additive ones. 
  
Next, the operator $\phi$ of multiplication by a matrix is homogeneous, and both 
the immersion $\iota$ and its inverse mapping $\iota^{-1}$ are positively homogeneous 
due to Proposition~\ref{HomogePropo}. Therefore, the composition $\;\iota\circ\phi 
\circ\iota^{-1} $ is at least positively homogeneous. In addition, for any 
$x\in\mbb{R}^{2n}$, there holds
\begin{align*}
(\,\iota\circ\phi\circ\iota^{-1}) (-x) \ & = \ (\,\iota\circ\phi\,) 
 (\opp\iota^{-1}(x)) \qquad \mbox {by \eqref{Opposite2}}      \\[2mm]
& = \ \iota\,\bigl(\, \opp (\,\phi\circ\iota^{-1}) (x) \,\bigr) \qquad
\mbox{\begin{tabular}{c}
      according\\
      with definitions \\
      \eqref{ScalarMult} and \eqref{Opposite}
      \end{tabular}} \\[2mm]
& = \ -(\,\iota\circ\phi\circ\iota^{-1}) (x) \qquad
 \mbox{due to \eqref{Opposite1}},
\end{align*}
which proves the homogeneity of the induced mapping $\;\iota\circ\phi\circ\iota^{-1}$ 
with respect to multiplication on any scalars. 
  
The second statement of the Theorem is a consequence of the definition of the standard 
immersion \eqref{StandardImmersion} and the rule of multiplying a number by the interval 
\begin{equation*}
q\cdot\mbf{x} 
 \ = \
\begin{cases}
\ [\,q\,\un{\mbf{x}}, q\,\ov{\mbf{x}}\,],& \mbox{ if }\ q\geq 0,\\[2mm]
\ [\,q\,\ov{\mbf{x}}, q\,\un{\mbf{x}}\,],& \mbox{ otherwise },
\end{cases}
\end{equation*}
for which we can give the following convenient equivalent form 
\begin{equation*}
\left\{\
\arraycolsep 2pt
\begin{array}{rcl}
\un{\, q\cdot\mbf{x} \,}
 &=& \phantom{-} q^+\,\un{\mbf{x}} - q^-\,\ov{\mbf{x}},\\[2mm]
\ov{\, q\cdot\mbf{x} \,}
 &=& - q^-\,\un{\mbf{x}} + q^+\,\ov{\mbf{x}}.
\end{array}
\right.
\end{equation*} 
\end{proof}
  
The block $2n\!\times\!2n$-matrix from Theorem~\arabic{MatrixTheo} is so important 
in the theory we develop that we will take a special notation and term for it. 
  
\begin{definition}
\label{ExtCoeffMatrix}
If $Q$ is a point $n\times n$-matrix, we set 
\begin{equation*}
\arraycolsep 3mm
\renewcommand{\arraystretch}{1.5}
Q\tl :=
\left( \begin{array}{c|c}
Q^+ & Q^- \\  \hline
Q^- & Q^+
\end{array} \right)
\tag{\ref{BlockForm}}
\end{equation*}
and will refer to the point $2n\times 2n$-matrix $Q\tl$ as \textsl{extended 
multiplier matrix} for $Q$. 
\end{definition}
  
An important feature of the extended multiplier matrix $Q\tl\in\mbb{R}^{2n\times 2n}$ 
is that it is always \emph{non-negative}: such a matrix must match a ``$\leq$''-isotonic 
operator on $\mbb{R}^{2n}$ induced by multiplying by $Q$ in the interval space 
$\mbb{KR}^n$, which is inclusion isotonic.

\begin{corollary} from Theorem~\arabic{MatrixTheo}. 
Using the definition of the induced map, it is easy to conclude that, for any point 
$n\times n$-matrix $Q$ and any $x\in\mbb{R}^{2n}$, there holds 
\begin{equation}
\label{SigmaProperty2}
\sti\bigl(\,Q\cdot\sti^{-1}(x)\,\bigr) = Q\tl x.
\end{equation} 
Similarly, for any point $n\times n$-matrix $Q$ and any interval vector 
$\mbf{x}\in\mbb{KR}^n$, we have 
\begin{equation}
\label{SigmaProperty3}
Q\mbf{x} = \sti^{-1}\bigl(\,Q\tl\cdot\sti(\mbf{x})\,\bigr)
\end{equation} 
(The relations \eqref{SigmaProperty2} and \eqref{SigmaProperty3} are illustrated 
by the commutative diagram shown in Fig.~\ref{ComDiagram}). 
\end{corollary}

  
\subsection{Absolutely regular matrices} 
\label{AbsRegularSubse}

\begin{theorem}
\setcounter{AbsRegTheo}{\value{theo}}
For a point $n\times n$-matrix $Q$, the following conditions are equivalent: 
\begin{itemize}
\item[\normalfont(a)] for any $x\in\mbb{KR}^n$, $Q\mbf{x} = 0$ \,if and only 
  if\, $\mbf{x} = 0$; 
\item[\normalfont(b)] the matrix $Q\tl\in\mbb{R}^{2n\times 2n}$, extended multiplier 
  matrix for $Q$, is regular; 
\item[\normalfont(c)] both the matrix $Q$ and its absolute value $|Q|$ 
  (i.\,e. the matrix \\*  formed by moduli of the elements of $Q$) are regular. 
\end{itemize}
\end{theorem}
  
\begin{proof}
The equivalence ``(a) $\Leftrightarrow$ (b)'' is a consequence of representation 
\eqref{SigmaProperty3}. 
  
To prove the equivalence of conditions (b) and (c), we perform the following 
transformations with the extended multiplier matrix $Q\tl$. We add its first row 
to the $(n+1)$-th row, then we add its second row to the $(n+2)$-th row, and so on, 
up to the $n$-th row, which we add the $2n$-th row of the matrix $Q\tl$. Insofar as 
\begin{equation*} 
q^+ + q^- = |q| 
\end{equation*} 
for any real number $q$, then the result of our transformation will be the following 
$2n\times 2n$-matrix: 
\begin{equation}
\label{ModulForm}
\arraycolsep 2mm
\renewcommand{\arraystretch}{1.5}
\left( \begin{array}{ll}
\ \phantom{i}Q^+ & \phantom{i}Q^- \\
\ |Q| & |Q|
\end{array} \right).
\end{equation}
  
Next, we subtract the $(n+1)$-th column of the matrix \eqref{ModulForm} from its first 
column, then we subtract the $(n+2)$-th column from its second column, and so on, up to 
the $2n$-th column which we subtract from the $n$-th column of \eqref{ModulForm}. Since 
\begin{equation*} 
q^+ - q^- = q
\end{equation*} 
for any real number $q$, then we get a block-triangular $2n\times 2n$-matrix 
\begin{equation}
\label{OrigForm}
\arraycolsep 2mm
\renewcommand{\arraystretch}{1.5}
\left( \begin{array}{ll}
\ Q & \phantom{i}Q^- \\[1mm]
\ \raisebox{-0.3ex}{\Large 0} & |Q|
\end{array} \right).
\end{equation}
  
As is known from matrix theory, the transformations we have done (linear combination 
of rows and columns) do not change the property of a matrix to be regular or singular 
(see, for example, \cite{Gantmacher}). Consequently, the matrix \eqref{OrigForm} is 
regular or singular along with the extended multiplier matrix $Q\tl$. But due to 
the special form of the matrix \eqref{OrigForm}, its determinant is equal to the product 
of the determinants of the diagonal blocks, i.\,e. determinants of the matrices $Q$ and 
$|Q|$. Thus, the extended multiplier matrix $Q\tl$ is regular indeed if and only if 
both matrices $Q$ and $|Q|$ are regular (non-singular). 
\end{proof}
  
We have already noted that the regularity (non-singularity) of the point matrix $Q$ 
in the sense of classical linear algebra does not necessarily imply that the corresponding 
multiplication operator by $Q$ in $\mbb{KR}^n$ is invertible. But now the phenomenon of 
the matrix \eqref{BadMatrix} and its like is fully explained: although such matrices 
themselves may be non-singular, but multiplying by them corresponds, after immersion 
in linear space, to multiplication by singular extended multiplier matrices. For example, 
for the matrix \eqref{BadMatrix}, the extended multiplier matrix is 
\begin{equation}
\label{ConcomBadMatrix}
\left(
\begin{array}{cccc}
1 & 1 & 0 & 0\\
0 & 1 & 1 & 0\\
0 & 0 & 1 & 1\\
1 & 0 & 0 & 1
\end{array}
\right),
\end{equation}
and its determinant is zero!

The result of Theorem~\arabic{AbsRegTheo} allows us to state the following 
  
\begin{definition} 
\setcounter{AbsRegDefi}{\value{defi}}
A point $n\times n$-matrix $Q$ that satisfies any one (and, hence, every one) of 
the equivalent conditions \textnormal{(a)--(c)} of Theorem~{\rm\arabic{AbsRegTheo}}, 
is called \textsl{absolutely regular} (\textsl{absolutely non-singular}). 
\end{definition}
  
For example, the identity matrix is absolutely regular, whereas the matrix 
\eqref{BadMatrix} is regular in the usual sense, but not absolutely regular. 
It is also obvious that if a matrix is singular in the usual sense, then it is 
certainly not absolutely regular. All non-negative regular matrices are also 
absolutely regular. A practically convenient criterion for checking the absolute 
regularity of a matrix is provided by condition (c) from Theorem~\arabic{AbsRegTheo}. 
For example, instead of calculating the determinant of the extended multiplier 
matrix \eqref{ConcomBadMatrix} for the matrix \eqref{BadMatrix}, one could just 
notice that the moduli matrix 
\begin{equation*}
\left(
\begin{array}{rr}
1 & 1\\
1 & 1
\end{array}
\right) 
\end{equation*}
is singular. 
  
The class of matrices, determined by the condition of Definition~\arabic{AbsRegDefi} 
and equivalent to it, was introduced by the author in the works of the past century 
\cite{Shary-95, Shary-98}. But the term ``absolutely regular matrix'' was coined 
relatively recently, already in the 2000s. Previously, these matrices were called 
\emph{$\iota$-regular} (\emph{$\iota$-nonsingular}), and then \emph{completely regular} 
(\emph{completely nonsingular}) matrices \cite{SharySurvey}. A long and nontrivial way 
indeed, but the term ``absolutely regular matrix'' seems to be the most adequate to 
the meaning of the concept. 
  
\bigskip
\begin{corollary} from Theorems~\arabic{MatrixTheo} and \arabic{AbsRegTheo}. 
The operator $\phi:\mbb{KR}^n\to\mbb{KR}^n$ determined by the multiplication 
by a square point matrix in $\mbb{KR}^n$, i.\,e. such that 
\begin{equation*} 
\phi(\mbf{x}) = Q\mbf{x} \qquad
\mbox{ for some }\ Q\in\mbb{R}^{n\times n},
\end{equation*} 
is invertible if and only if the matrix $Q$ is absolutely regular. Then the inverse 
operator $\phi^{-1}: \mbb{KR}^n \to\mbb{KR}^n$ acts, according to \eqref{InvMap}, 
as follows: 
\begin{equation}
\label{InverseOperator}
\phi^{-1}(\mbf{x}) \ = \ \sti^{-1}\bigl((\,Q\tl)^{-1}\cdot\sti\,(\mbf{x})\bigr).
\end{equation}
\end{corollary}

\bigskip
Despite the existence of the explicit formula \eqref{InverseOperator}, an operator 
that is inverse to the operator of multiplication by a point $n\times n$-matrix $Q$ 
in $\mbb{KR}^n$ cannot generally be expressed through multiplication by a matrix from 
$\mbb{KR}^n$ (in particular, by the matrix $Q^{-1}$). This follows, for example, from 
the fact that the matrix $(Q\tl)^{-1} $ does not need to be non-negative, and then 
the inverse operator is not monotonic with respect to inclusion in $\mbb{KR}^n$. 
But multiplication by a point matrix in $\mbb{KR}^n $ is always monotone with respect 
to inclusion. 
  
The above Corollary is formulated in an abstract manner, but it has quite practical 
appli\-cations. Using formula \eqref{InverseOperator}, one can easily find formal 
solutions to interval systems of linear equations in which the square matrix is 
a point absolutely regular matrix, and only the right-hand side vector is interval, 
i.\,e. of the form 
\begin{equation} 
\label{PointILinSys} 
Ax = \mbf{b} \qquad A\in\mbb{R}^{n\times n}, \;\mbf{b}\in\mbb{KR}^n. 
\end{equation} 
  
Consider, for example, the interval linear system 
\begin{equation} 
\label{SamplePointSys} 
\left( 
\begin{array}{@{\;}cc@{\;}} 
 1 & 2 \\[2mm] 
-3 & 4 
\end{array} 
\right) 
\,
\begin{pmatrix}
x_{1} \\[2mm] x_{2} 
\end{pmatrix} 
= 
\left( 
\begin{array}{@{\;}c@{\;}} 
{[0, 10]} \\[2mm] 
{[10, 20]} 
\end{array} 
\right) 
\end{equation} 
First of all, we note that the matrix of this system is point (non-interval) and 
absolutely regular, since both the matrix itself and its modulus are regular: 
\begin{equation*} 
\det\left( 
\begin{array}{@{\;}cc@{\;}} 
 1 & 2 \\[2mm] 
-3 & 4 
\end{array} 
\right) 
\neq 0 
\qquad\text{ and }\qquad 
\det\left( 
\begin{array}{@{\;}cc@{\;}} 
 1 & 2 \\[2mm] 
 3 & 4 
\end{array} 
\right) 
\neq 0. 
\end{equation*} 
Then the multiplication operator by this matrix is invertible in the interval space 
$\mbb{KR}^n$, and its inverse is given by formula \eqref{InverseOperator}. To apply 
it for computing formal solutions to \eqref{SamplePointSys}, we construct the extended 
multiplier matrix 
\begin{equation*} 
A\tl = 
\left( 
\begin{array}{@{\;}cc@{\;}} 
A^{+} & A^{-} \\[4pt]
A^{-} & A^{+} 
\end{array} 
\right) 
= 
\left( 
\begin{array}{@{\;}cc|cc@{\;}} 
 1 & 2 & 0 & 0 \\[3pt] 
 0 & 4 & 3 & 0 \\ 
 \hline 
 0 & 0 & 1 & 2 \\[3pt] 
 3 & 0 & 0 & 4 
\end{array} 
\right), 
\end{equation*} 
and then multiply its inverse by the result of the standard immersion ``sti'' 
acting on the right-hand side vector of \eqref{SamplePointSys}, that is, 
by $(0, -10, 10, 20)^\top$: 
\begin{equation*} 
\bigl(A\tl\bigr)^{-1} 
\left( 
\begin{array}{@{\;}c@{\;}} 
0\\[2pt] -10\\[2pt] 10\\[2pt] 20
\end{array} 
\right) 
= 
\left( 
\begin{array}{@{\;}c@{\;}} 
-4\\[2pt]  2\\[2pt] -6\\[2pt]  8
\end{array} 
\right) 
\end{equation*} 
Then the formal solution for system \eqref{SamplePointSys} is 
\begin{equation*} 
\sti^{-1} 
\left( 
\begin{array}{@{\;}c@{\;}} 
-4\\[2pt]  2\\[2pt] -6\\[2pt] 8 
\end{array} 
\right) = 
\left( 
\begin{array}{@{\;}c@{\;}} 
{[4, -6]}\\[3pt]  {[-2, 8]} 
\end{array} 
\right) 
\end{equation*} 
By direct substitution, we can verify that the resulting interval vector is indeed 
a formal solution of system \eqref{SamplePointSys}. 
  
S.\,Markov proposed in \cite{Markov} an equivalent method for finding formal solutions 
to interval linear systems with point (non-interval) matrices, i.\,e. of the form 
\eqref{PointILinSys}. Its essence, if we give up unnecessary abstractions, is 
as follows. 
  
If $\mbf{x}^\ast\in\mbb{KR}^n$ is a formal solution to system \eqref{PointILinSys}, 
then, taking the midpoint and radius of both sides of the equality 
\begin{equation*} 
A\mbf{x}^\ast = \mbf{b}, 
\end{equation*} 
we get 
\begin{eqnarray*} 
&& A\cdot\m\mbf{x}^\ast = \m\mbf{b}, \\[3pt] 
&& |A|\cdot\r\mbf{x}^\ast = \r\mbf{b}. 
\end{eqnarray*} 
Thus, computing the formal solution is reduced to the solution of two ordinary systems 
of linear equations with the matrices $A$ and $|A|$, from which we find the middle of 
the formal solution $\mbf{x}^\ast$ and its radius. It is clear that a sufficient 
condition for the existence of these solutions is the non-singularity of both matrices 
$A$ and $|A|$. If we want the solvability of these systems of equations for the middle 
and radius of the formal solution for any right-hand sides $\mbf{b}$, then the matrix 
$A$ must be absolutely regular.

  
\section{The subdifferential Newton method}

In this section of the paper, we briefly present results on the \emph{subdifferential 
Newton method} for finding formal solutions of interval linear systems. 
  
Recall some definitions and facts from convex analysis \cite{LemarHirUrr,Rockafellar}. 
Let $U, V$ be real linear spaces, and $V$ be an ordered linear space with an order 
``$\preceq$''. The mapping $F: U\rightarrow V$ is called \emph{order convex} with 
respect to ``$\preceq$'', if 
\begin{equation*} 
F(\lambda x + (1-\lambda)y)\preceq\lambda F(x) + (1-\lambda) F(y)
\end{equation*} 
for any $x,\,y\in U$ and $\lambda\in\;]0,1[$. \emph{Subdifferential} of the order convex 
mapping $F: U\rightarrow V$ at the point $x$ is the set $\partial_\preceq F(x)$ of all 
linear operators $D: U\rightarrow V$ such that 
\begin{equation*} 
D(y-x)\preceq F(y) - F(x)
\end{equation*} 
for any $y\in U$. It is known that for each interior point $x$ from the effective domain 
of definition of the order convex function $F$, the subdifferential is a non-empty set, 
whose elements are called \emph{subgradient} of $F$ at $x$. 
  
Suppose that the interval matrix $\mbf{A}$ is such that in each of its rows all elements 
are either only proper or only improper, and let the sets of natural numbers $I '= \{\, 
i'_{1}, i'_{2}, \ldots, i'_{\alpha}\,\}$ and $I'' = \{\, i''_{1}, i''_{2}, \ldots, 
i''_{\beta} \,\}$, $I'\cap I'' = \varnothing$, $\alpha + \beta = n$, represent row 
indices of $\mbf{A}$ with proper and improper intervals, respectively: 
\begin{equation*} 
\mbf{a}_{ij} \ \mbox{ is } \
\left\{ \
\begin{array}{ll}
\mbox{ proper,}  &\mbox{if} \ i\in I' \\[1mm]
\mbox{ improper,}&\mbox{if} \ i\in I''.
\end{array}
\right.
\end{equation*} 
If we define on $\mbb{R}^{2n}$ a partial order ``$\sqsubseteq$'' as follows 
\begin{equation*} 
x\sqsubseteq y \ \Longleftrightarrow \
\left\{ \
\begin{array}{ll}
\ x_i\leq y_i & \mbox{for} \  i\in\{\,i'_1, \ldots, i'_\alpha,
  i'_{1} + n, \ldots, i'_{\alpha} + n\,\}, \\[1mm]
\ x_i\geq y_i & \mbox{for} \  i\in\{\,i''_1, \ldots, i''_\beta,
  i''_{1} + n, \ldots, i''_{\beta} + n\,\},
\end{array}
\right.
\end{equation*} 
then, as a consequence of the sub- and superdistributive properties of Kaucher 
arithmetic, the induced mapping $\Phi(x)$, determined by \eqref{InducedMapping} for 
the interval mapping $\mbf{x}\mapsto\mbf{Ax}\ominus\mbf{b}$, is order convex with 
respect to ``$\sqsubseteq$''. Accordingly, at any point $x\in\mbb{R}^{2n}$ there will 
be a non-empty subdifferential $\partial_\sqsubseteq\Phi(x)$, easily computable, 
because $\Phi$ is a polyhedral (piecewise affine) mapping. The \emph{subdifferential 
Newton method} proved to be an efficient tool for finding formal solutions to such 
interval linear algebraic systems. It is a further development of the well-known 
results on monotonically convergent Newton-type methods in ordered linear spaces. 
  
\bigskip
\begin{center}
\noindent{\sf The subdifferential Newton method with a special starting vector}\\[2mm]
\fbox{ \
\centering
\begin{minipage}{150mm}
\medskip\par
As a starting approximation of $x^{(0)}\in\mbb{R}^{2n}$, we take the solution of 
the ``midpoint'' $2n\!\times\!2n$-system 
\begin{equation*} 
(\m\mbf{A})\tl\,x = \sti(\mbf{b}).
\end{equation*} 
If the $k$-th approximation $x^{(k)}\in\mbb{R}^{2n}$, $k = 0,1, \ldots\,$, is already 
known, then we find some subgradient $D^{(k)}\in\partial_\sqsubseteq \Phi(\,x^{(k)})\,$, 
$D\in\mbb{R}^{2n\times 2n}$, and then take 
\begin{equation*} 
x^{(k+1)} \gets x^{(k)} - \tau\left(\,D^{(k)}\right)^{-1} 
\left(\Phi(\,x^{(k)})\right).
\bigskip
\end{equation*} 
\end{minipage} \ }
\end{center}
\bigskip
  
In the above pseudocode, $\tau$ is a damping factor from $]0, 1]$. There holds 
  
\begin{theorem} 
If, in each row of the interval matrix $\mbf{A}$, all elements are only proper 
or only improper, the proper projection $\pro\mbf{A}$ contains only absolutely 
regular point matrices and is sufficiently narrow (i.\,e. $\|\r\mbf{A}\|$ is small 
enough), then the subdifferential Newton method converges in $\mbb{R}^{2n}$ to 
$\sti(\mbf{x}^*)$, where $\mbf{x}^*$ is a formal solution of the system 
$\mbf{Ax} = \mbf{b}$ and ``sti'' is the standard immersion of $\mbb{KR}^n$ 
to $\mbb{R}^{2n}$. 
\end{theorem}
  
As applied to more particular types of interval systems, the proof of the convergence 
of the subdifferential Newton method and its detailed discussion are contained in 
\cite{Shary-98}. In practice, the subdifferential method of Newton converges extremely 
quickly, in a small finite number of iterations (especially for $\tau$ close or equal 
to one), which is explained by the polyhedral nature of the induced function $\Phi(x)$ 
that corresponds to interval linear systems. Another advantage of this method is the 
absence of any problems with the choice of the starting vector. Finally, the subdifferential 
Newton method works well in practice even for such interval systems in which proper and 
improper elements are arbitrarily mixed in the interval matrix. In general, the Newton 
subdifferential method is of great practical importance, but in this paper we will 
consider it mainly as an auxiliary algorithm, the initial link in the composition 
of numerical processes with a wider domain of convergence.

  
\section{Stationary single-step iterative methods} 
  
  
\subsection{General approach: distributive splittings} 
\label{GenSplitSect}

How else can numerical methods be proposed for finding formal (algebraic) solutions 
to interval linear systems of equations \eqref{InLinSys1}--\eqref{InLinSys2} with 
arbitrary matrices and right-hand sides? 
  
Although interval arithmetic differs from the real axis $\mbb{R}$ on its algebraic 
properties, the simple structure of system \eqref{InLinSys1}--\eqref{InLinSys2} 
makes it possible to apply for this purpose some general ideas related to stationary 
iterative methods for ordinary linear systems of equations. In particular, we can 
construct interval versions of Jacobi method, Gauss-Seidel method, etc. When 
implementing such methods, it is sometimes possible to iterate directly in the 
interval space $\mbb{IR}^n$ or $\mbb{KR}^n$, even without immersing it in the linear 
space of double dimension $\mbb{R}^{2n}$. On the other hand, the convergence obtained 
in such stationary iteration algorithms is principally slower than that for 
the subdifferential (quasidifferential) Newton method. 
  
It should be noted that iterative Jacobi-type methods for finding formal solutions 
to interval linear systems of equations have already been considered by different 
authors. V.\,Zyuzin was a pioneer in this area, who proposed them in the late 80s 
of the last century \cite{Zyuzin-87,Zyuzin-89,Zyuzin-90}. Then these methods were used 
in the works of L.\,Kupriyanova \cite{Kupriyanova}, S.\,Markov \cite{Markov,MarPopUll}, 
and Spanish researchers \cite{SaiGarJor,ModalIntAnal}. But the applicability scope 
and convergence rate of these simplest methods are unsatisfactory, which necessitates 
the development of more advanced numerical methods. Below, we consider a more general 
construction than Jacobi-type methods. 
  
In accordance with the general scheme of stationary single-step iterative methods, 
the original equation \eqref{InLinSys1}--\eqref{InLinSys2} must be equivalently 
reduced to a fixed-point form 
\begin{equation}
\label{FixedPointForm}
\mbf{x} = \mbf{T}(\mbf{x})
\end{equation}
with some operator $\mbf{T}: \mbb{KR}^n\rightarrow\mbb{KR}^n$. Then, after 
selecting an initial approximation $\mbf{x}^{(0)}$, iterations are started: 
\begin{equation}
\label{Iterations} 
\mbf{x}^{(k+1)} \gets\mbf{T} (\,\mbf{x}^{(k)}),
   \qquad  k = 0, 1, 2, \ldots.
\end{equation}
  
Under some special conditions on the transition operator $\mbf{T}$ (when it is 
a contraction, etc.) and on the initial approximation $\mbf{x}^{(0)}$, the sequence 
$\mbf{x}^{(k)}$ converges to a fixed point of $\mbf{T}$, i.\,e. to the desired formal 
solution of the equation \eqref{InLinSys2}. But, in contrast to the traditional 
non-interval case, bringing the interval system \eqref{InLinSys1}--\eqref{InLinSys2} 
to the form \eqref{FixedPointForm} is not a trivial task due to insufficient algebraic 
properties of Kaucher arithmetic. The problem is that at least \emph{two} terms with 
the variable $\mbf{x}$ in the formula \eqref{FixedPointForm} (which is equivalent to 
$\,\mbf{x} \ominus\mbf{T}(\mbf{x}) = 0\,$) should eventually collapse into the expression 
$\,\mbf{Ax}\ominus\mbf{b}$, containing only \emph{one} occurrence of the variable 
$\mbf{x}$. In the absence of a full-fledged opportunity to reduce such terms, 
special means are required for transforming the original equation \eqref{InLinSys2} 
into the form \eqref{FixedPointForm}. 
  
One of the promising approaches to the construction of iterative schemes 
for the solution of the equation \eqref{InLinSys2} is to go the ``opposite way'', 
from possible representations 
\begin{equation}
\label{Splitting}
\mbf{Ax} = \eus{G} (\mbf{x}) + \eus{H}(\mbf{x}),
\end{equation}
where\par
\hspace*{12mm}\parbox{144mm}{%
  \begin{itemize}
  \item[(i)] 
    the function $\eus{G}:\mbb{KR}^n\to\mbb{KR}^n$ is ``easily invertible'', 
    i.\,e. its inverse function $\eus{G}^{-1}: \mbb{KR}^n\to\mbb{KR}^n$, such that 
    $\eus{G}^{-1} (\,\eus{G}(\mbf{x})\,) = \ = \eus{G} (\,\eus{G}^{-1} (\mbf{x})\,) 
    = \mbf{x}$, can be easily constructed for $\eus{G}$; 
  \item[(ii)] 
    the function $\eus{H}: \mbb{KR}^n\to\mbb{KR}^n$ is easily computable. 
  \end{itemize}}
  
\begin{definition} 
\setcounter{SplitDefi}{\value{defi}} 
Let $\psi: \mbb{KR}^n \to\mbb{KR}^n$ be an operator in $\mbb{KR}^n$ determined 
by multiplication by the interval matrix $\mbf{A}$, i.\,e. such that $\,\psi(\mbf{x}) 
= \mbf{Ax}$. The representation of $\psi$ in the form of \eqref{Splitting}, satisfying 
for any $\mbf{x}\in\mbb{KR}^n$ the above conditions {\rm(i)--(ii)} will be called 
a \textsl{splitting of the operator~$\psi$ of multiplication by the matrix $\mbf{A}$}, 
or, briefly, \textsl{splitting of the matrix $\mbf{A}$}. 
\end{definition}

If we know some splitting of matrix $\mbf{A}$ in the interval linear system 
\begin{equation*}
\mbf{A}x = \mbf{b},
\tag{\ref{InLinSys2}}
\end{equation*}
then we can proceed to the equivalent equation 
\begin{equation*}
\eus{G}(\mbf{x}) + \eus{H}(\mbf{x}) = \mbf{b}, 
\end{equation*}
or 
\begin{equation*}
x = \eus{G}^{-1}\bigl(\,\mbf{b}\ominus \eus{H}(x)\,\bigr),
\end{equation*}
which coincides with the desired form \eqref{FixedPointForm}. Accordingly, 
the iterative process for computing the formal solution can be organized by the formula 
\begin{equation}
\label{SplitForm}
\mbf{x}^{(k+1)} \gets
\eus{G}^{-1}\,\bigl(\,\mbf{b}\ominus\eus{H}(\mbf{x}^{(k)})\,\bigr),
\qquad k = 0, 1, 2, \ldots\,,
\end{equation} 
after chosing a starting approximation $\mbf{x}^{(0)}$. Further, we restrict ourselves 
to the simplest case where the functions $\eus{G}$, $\eus{H}: \mbb{KR}^n \to\mbb{KR}^n$ 
themselves either represent multiplications by some interval matrices, or they are mappings 
similar in form. 
  
Suppose that, in system \eqref{InLinSys2}, the interval matrix $\mbf{A}$ is regular. 
Then there are at least two possibilities for the splitting \eqref{Splitting}, i.\,e. 
for the splitting of the multiplication operator by $\mbf{A}$, with an easily 
invertible mapping $\eus{G}^{-1}(\cdot)$: 
\begin{itemize}
\bigskip
\item[A)]
$\eus{G}(\cdot)$ is taken in the form of multiplication by a point absolutely regular 
matrix $G$, that is, 
\begin{equation*}
\eus{G}(\mbf{x}) = G\mbf{x}.
\end{equation*} 
In this case, $\,\eus{H}(\mbf{x}) = \mbf{Hx}$, $\mbf{H} = \mbf{A} - G$, and nonzero 
elements of $G$ and $\mbf{H}$ are chosen so that they have the same signs (due 
to this fact, equality \eqref{Splitting} is ensured for all $\mbf{x}, \mbf{y}\in
\mbb{KR}^n$ by virtue of distributivity \eqref{KRSignDistr}). The inverse 
mapping $\eus{G}^{-1}(\cdot)$ is then determined according to Corollary 
from Section~\ref{AbsRegularSubse}. 
\bigskip
\item[B)]
$\eus{G}(\cdot)$ and $\eus{H}(\cdot)$ are taken, respectively, in the form 
\begin{equation*}
\eus{G}(\mbf{x}) = \mbf{Gx}
\quad\mbox{ and }\quad
\eus{H}(\mbf{x}) = \mbf{Hx},
\end{equation*}
where $\mbf{G}$ and $\mbf{H}$ are upper and lower triangular interval matrices, 
$\mbf{A} = \mbf{G} + \mbf{H}$, where $\mbf{G}$ has invertible elements on the main 
diagonal, and $\mbf{H}$ has zero main diagonal (perhaps, to do this, one first need 
to swap the equations of the system).   \par 
Then the inverse mapping of $\eus{G}^{-1}(\cdot)$ is such that the result $\mbf{y}$ 
of its action on the element $\mbf{x}\in\mbb{KR}^n$ is determined by the ``backward 
substitution'' formulas for the triangular system $\mbf{Gy = x}$. It is natural 
to use the term \emph{triangular splitting} with respect to such a splitting 
of the multiplication operator by $\mbf{A}$. 
\end{itemize}
  
\bigskip
Note that in both cases the inverse mapping $\eus{G}^{-1}: \mbb{KR}^n \to\mbb{KR}^n\,$, 
generally speaking, cannot be defined through multiplication by an interval matrix.

  
\subsection{Splitting an absolutely regular matrix}

In this section, for the first of the matrix splitting cases considered 
in \S\ref{GenSplitSect}, we write out the computation formulas of the corresponding 
iterative process \eqref{SplitForm} in an explicit form. We also show how it is possible 
in practice to construct a splitting of an interval matrix. 
  
\begin{definition}
For $\mbf{a}\in\mbb{KR}$, we denote 
\begin{equation*}
\lfloor\mbf{a}\rfloor :=
\begin{cases}
\ \max\{\,\un{\mbf{a}},\,\ov{\mbf{a}}\,\}, 
                & \text{if \ $\mbf{a}<0$},\\[3pt]
\qquad 0\hfill, & \text{if \ $0\in\pro\mbf{a}$},\\[3pt]
\ \min\{\,\un{\mbf{a}},\,\ov{\mbf{a}}\,\}, 
                & \text{if \ $\mbf{a}>0$},
\end{cases}
\end{equation*}
that is, taking the element which is closest to zero from the proper projection 
of the interval $\mbf{a}$. 
\end{definition}
  
$\lfloor\mbf{a}\rfloor$ is the point with the \emph{smallest} absolute value from 
the proper  projection of the interval, and it has the same sign as the interval 
itself. It is easy to understand that if a real number $a$ lies between $0$ and 
$\lfloor\mbf{a}\rfloor$, i.\,e. $a\in 0\lor\lfloor\mbf{a}\rfloor$, the intervals 
$(\mbf{a} - a)$ and $[a, a]$ have the same sign, and therefore constitute 
a splitting for the multiplication operator on $\mbf{a}$ in the sense of 
Definition~\arabic{SplitDefi}. Therefore, the condition~A of \S\ref{GenSplitSect} 
can be satisfied, for example, if we take 
\begin{equation} 
\label{DSCondition} 
G\in 0\lor\lfloor\mbf{A}\rfloor,
\end{equation}
where $\lfloor\mbf{A}\rfloor$ means element-wise application to $\mbf{A}$ of 
the unary operation $\lfloor\cdot\rfloor$, and the interval matrix in the right-hand 
side of the inclusion is obtained by using the operation \eqref{Supremum}. 
  
To minimize the absolute value of the remainder $\mbf{H} = \mbf{A} - G$, one can 
assign $G = \lfloor\mbf{A}\rfloor$. It~is clear that, under the assumption we made 
about the regularity of $\mbf{A}$, the matrix $G$ also turns out to be non-singular. 
If, additionally, $G$ is absolutely regular, then the inverse mapping $\eus{G}^{-1} 
(\cdot)$ corresponds to multiplication by the matrix $(G\tl)^{-1}$ in $\mbb{R}^{2n}$ 
(see Section~\ref{AbsRegularSubse}). In any case, we always have the opportunity 
to make the matrix $G$ absolutely regular by slightly reducing the absolute value 
of its nonzero elements that do not violate the splitting condition \eqref{DSCondition}. 
  
We write out the formulas of the resulting iteration process in the Euclidean 
space $\mbb{R}^{2n}$. As~a~specification for \eqref{SplitForm}, we get 
\begin{equation}
\label{1stForm}
x^{(k+1)}\gets (G\tl)^{-1}\;\sti\!\bigl(\,\mbf{b}\ominus\mbf{H}\, 
  \sti^{-1}(\,x^{(k)})\bigr), 
\end{equation}
where ``sti'' is the standard immersion of $\mbb{KR}^n$ into $\mbb{R}^{2n}$ and 
\begin{equation}
\label{SimplestSplit}
G\in\mbb{R}^n, \qquad\mbf{H} = \mbf{A} - G.
\end{equation}
  
The iterative process with such splitting works satisfactorily, but sometimes not as 
well as one would like. For example, it does not lead to success in solving the 
Barth-Nuding interval linear system \cite{BarNud}, 
\begin{equation}
\label{BarthNudingSystem} 
\left( 
\begin{array}{cc}
[2, 4] & [-2, 1] \\[2mm] 
[-1, 2] & [2, 4] 
\end{array} 
\right) 
\,x = 
\left( 
\begin{array}{c}
[-2, 2] \\[2mm] [-2, 2] 
\end{array}
\right). 
\end{equation}
Therefore, it makes sense to consider other recipes for splitting the matrix of 
the interval linear system. 
  
Yet another way of splitting the interval matrix can be based on the generalized 
distributivity law \eqref{MarkovDistr} suggested by S.\,Markov. We introduce

\begin{definition}
For $\mbf{a}\in\mbb{KR}$, we denote 
\begin{equation*}
\lceil\mbf{a}\rceil =
\begin{cases}
\ \min\{\,\un{\mbf{a}},\,\ov{\mbf{a}}\,\}, 
                & \text{if \ $\mbf{a}\leq 0$},\\[3pt]
\qquad 0\hfill, & \text{if \ $0\in\pro\mbf{a}$},\\[3pt]
\ \max\{\,\un{\mbf{a}},\,\ov{\mbf{a}}\,\}, 
                & \text{if \ $\mbf{a}\geq 0$}, 
\end{cases}
\end{equation*} 
that is, taking the largest, in absolute value, element from the proper projection 
of the interval, if it does not contain zero, and zero otherwise. 
\end{definition}
  
If $0\not\in\pro\mbf{a}$, then $\lceil\mbf{a}\rceil$ is the point from the proper 
projection of the interval $\mbf{a}$, which has the \emph{largest} absolute value 
and the same sign as the interval itself (unlike $\lfloor\mbf{a}\rfloor$). It is 
easy to see that if a real number $a$ coincides in sign with $\lceil\mbf{a}\rceil$ 
and $|a| > \lceil\mbf{a}\rceil$, then the intervals $(\mbf{a} - a)$ and $[a, a]$ 
have different signs, and the sign of their sum $(\mbf{a} - a) + a$ is the same 
as for $\mbf{a}$. Therefore, for any $\mbf{x}\in\mbb{KR}$, by virtue 
of \eqref{MarkovDistr}, there holds 
\begin{equation*} 
\bigl(\,(\mbf{a} - a) + a\,\bigr)\cdot\mbf{x} =
(\mbf{a} - a)\cdot\dual\mbf{x} + a\cdot\mbf{x}.
\end{equation*} 
Consequently, in the general formula of the iterative processes \eqref{SplitForm}, 
we can put 
\begin{eqnarray}
\eus{G}(\mbf{x}) = G\mbf{x},
&&  G = (\,g_{ij}) = \lceil\mbf{A}\rceil , \label{gMap} \\[6mm]
\eus{H}(\mbf{x}) =
\begin{pmatrix}
\eus{H}_1(\mbf{x})\\[1mm]
\eus{H}_2(\mbf{x})\\[1mm]
     \vdots       \\[1mm]
\eus{H}_n(\mbf{x})
\end{pmatrix},
&& \eus{H}_i(\mbf{x}) =
 \sum_{j=1}^n \ \mbf{h}_{ij}\cdot
 \begin{cases}
 \ \mbf{x}_j, \text{ if } g_{ij} = 0,\\[1mm]
 \ \dual\mbf{x}_j , \text{ otherwise, }
 \end{cases} \phantom{AAA} 
 \label{hMap}\\[3mm]
&& \mbf{H} = (\,\mbf{h}_{ij}) = \lceil\mbf{A}\rceil - G.
 \label{MatrixH}
\end{eqnarray}
In $\mbb{R}^{2n}$, an explicit formula of the iterative process based on the above 
point splitting of the matrix of the interval system has the form 
\begin{equation} 
\label{2ndForm} 
x^{(k+1)}\gets (G\tl)^{-1}\; \sti\!\bigl(\,\mbf{b} 
                        \ominus\eus{H}(\sti^{-1}(\,x^{(k)}))\bigr), 
\end{equation} 
where the matrix $G\in\mbb{R}^{n\times n}$ and the mapping $\eus{H}(\cdot)$ are 
defined by \eqref{gMap}--\eqref{MatrixH}, and ``sti'' is the standard immersion of 
$\mbb{KR}^n$ in $\mbb{R}^{2n}$. Below, in Section~\ref{StatIterTests}, we present 
the results of numerical experiments with this method appearing under the name 
\texttt{ARMSplit} (derived from \un{A}bsolutely \un{R}egular \un{M}atrix 
\un{Split}ting), which show that it works significantly better than the process 
\eqref{1stForm}. 
  
What are the conditions for convergence of the considered iteration processes? 
The following result is valid: 
  
\begin{theorem} 
\setcounter{ARMSplitTheo}{\value{theo}} 
Let the matrices $G\in\mbb{R}^{n\times n}$ and $\mbf{H}\in\mbb{KR}^{n\times n}$ be 
the result of splitting, either \eqref{SimplestSplit} or \eqref{gMap}--\eqref{MatrixH}, 
applied to the interval matrix $\mbf{A}$, and $\eus{V}$ is the $2n\times 2n$-matrix 
$(G\tl)^{-1}$. If the spectral radius of the matrix $|\eus{V}|\,|\mbf{H}|\tl$ is less 
than one, 
\begin{equation*} 
\rho\bigl(|\eus{V}|\,|\mbf{H}|\tl\bigr) < 1, 
\end{equation*} 
then a formal solution $\mbf{x}^\ast$ of the interval linear system $\mbf{A}x = \mbf{b}$ 
exists and is unique, and iterations \eqref{1stForm} and \eqref{2ndForm} converge in 
$\mbb{R}^{2n}$ to $\sti(\mbf{x}^*)$, where ``{\,\rm sti}'' is the standard immersion of 
$\mbb{KR}^n$ to $\mbb{R}^{2n}$. 
\end{theorem} 
  
Before proceeding to the proof of the theorem, let us recall some facts and concepts. 
In order to consider convergence in the interval spaces and in usual Euclidean spaces, 
we need to have a ``distance'', a ``measure of the deviation'' of one vector from 
another. For ordinary linear spaces, distance can be introduced by various equivalent 
ways that are well known. The most popular is the concept of a metric (distance), as 
a non-negative real-valued function, which associates with each pair of vectors 
a non-negative number, the distance between them. 
  
The distance between intervals in the one-dimensional case is known to be defined as 
\begin{equation}
\label{VMetric}
\dist(\mbf{a},\mbf{b}) \;  := \;  |\mbf{a}\ominus\mbf{b}| \ 
   = \  \max\bigl\{|\un{\mbf{a}} - \un{\mbf{b}}|, |\ov{\mbf{a}} - \ov{\mbf{b}}|\bigr\}. 
\end{equation} 
In the same way, one can determine the distance in the multidimensional case, 
for interval vectors, but another construction is more popular. 
  
In multidimensional spaces, it is often convenient to work with a vector-valued 
metric (a vector-valued distance), often called \emph{multimetric}. For spaces 
of interval vectors $\mbb{IR}^n$ and $\mbb{KR}^n$, it is natural to introduce it as 
\begin{equation}
\label{VMultiMetric}
\Dist(\mbf{a}, \mbf{b})
\  :=  \ 
\left( 
\begin{array}{c}
\dist(\,\mbf{a}_1,\mbf{b}_1)\\[1mm]
\vdots\\[1mm]
\dist(\,\mbf{a}_n,\mbf{b}_n)
\end{array}
\right)
\in\mbb{R}^n_{+},
\end{equation}
that is, as a vector of distances between the components of the vectors $\mbf{a}$ 
and $\mbf{b}$. The multimetric is also defined in a similar way for ordinary linear 
vector spaces. It is easy to show that it has all the properties of a traditional 
scalar distance, i.\,e. non-negativity, symmetry, triangle inequality. With respect 
to multimetrics, we can talk about convergence, topological completeness and other 
important and necessary concepts.\footnote{Multimetric is called ``pseudometric'' 
in the popular book by L.\,Collatz \cite{Collatz}. We do not follow this usage 
because, in modern mathematics, the term ``pseudometric'' has a different meaning. 
It is ``a distance on a set that fulfils the same properties as a metric except 
relaxes the definition to allow the distance between two different points to be 
zero'' (a citation from MathWorld -- A Wolfram Web Resource).} 
  
Further, we will substantially rely on the following property of interval multimetrics. 
For any interval matrix $\mbf{A} = (\,\mbf{a}_{ij})\in\mbb{KR}^{m\times n}$ and any 
interval vectors $\mbf{x} = (\,\mbf{x}_{j})$, $\mbf{y} = (\mbf{y}_{j})\in\mbb{KR}^{n}$, 
there holds 
\begin{equation}
\label{MatrLipschIneq}
\Dist(\,\mbf{Ax},\mbf{Ay}\,) 
  \  \leq \  |\,\mbf{A}\,|\cdot\Dist(\,\mbf{x},\mbf{y}\,). 
\end{equation} 
Indeed, by virtue of the algebraic properties of interval operations, we can 
conclude that 
\begin{align*}
\dist\bigl(\,(\mbf{Ax})_{i},(\mbf{Ay})_{i}\bigr) \
&= \ \dist\left(\ \sum_{j=1}^n \mbf{a}_{ij}\mbf{x}_{j} \ ,
 \ \sum_{j=1}^n \mbf{a}_{ij}\mbf{y}_{j} \right)\\[3pt]
&\leq \ \sum_{j=1}^n \dist(\,\mbf{a}_{ij}\mbf{x}_{j},
 \mbf{a}_{ij}\mbf{y}_{j})\\[2pt]
&\leq \ \sum_{j=1}^n |\,\mbf{a}_{ij}|\,
 \cdot\dist(\,\mbf{x}_{j},\mbf{y}_{j}) 
\end{align*}
for all $i = 1,2,\ldots, m$, which proves the multidimensional 
estimate \eqref{MatrLipschIneq}.

\begin{proof} 
The convergence of the iterative processes \eqref{1stForm} and \eqref{2ndForm} 
in the Euclidean space $\mbb{R}^{2n}$ will be proved in some specially designed 
multimetric. Namely, we introduce a multimetric $\zeta: \mbb{R}^{2n}\to\mbb{R}^{2n}_+$ 
on $\mbb{R}^{2n}$ as follows: 
\begin{equation*}
\zeta (x,y) :=
\left(
\begin{array}{c}
\max\{\,| x_1 - y_1 |,| x_{n+1} - y_{n+1} |\,\} \\[4pt]
\vdots                                          \\[4pt]
\max\{\,| x_n - y_n |,| x_{2n} - y_{2n} |\,\}   \\[4pt]
\max\{\,| x_1 - y_1 |,| x_{n+1} - y_{n+1} |\,\} \\[4pt]
\vdots                                          \\[4pt]
\max\{\,| x_n - y_n |,| x_{2n} - y_{2n} |\,\}
\end{array}
\right).
\end{equation*} 
Recalling definition \eqref{StandardImmersion} of the standard immersion 
``sti'', we can determine the multi-metric $\zeta$ in another way, specifically 
\begin{equation*}
\zeta(x,y) =
\left(
\begin{array}{c}
|\,\sti^{-1}(x)\ominus\sti^{-1}(y)\,| \\[5pt]
|\,\sti^{-1}(x)\ominus\sti^{-1}(y)\,|
\end{array}
\right) =
\left(
\begin{array}{c}
\Dist\bigl(\,\sti^{-1}(x), \sti^{-1}(y)\,\bigr) \\[5pt]
\Dist\bigl(\,\sti^{-1}(x), \sti^{-1}(y)\,\bigr)
\end{array}
\right).
\end{equation*} 
  
The proof of Theorem~\arabic{ARMSplitTheo} will be first conducted for the iterative 
process \eqref{1stForm}. We show that, with respect to the above defined multi-metric 
$\zeta$, the transition operator $T$ of the iteration scheme \eqref{1stForm}, 
determined as 
\begin{equation*}
T(x)\; = \;\eus{V}\;\sti\bigl(\mbf{b}\ominus\mbf{H}\,\sti^{-1}(x)\bigl), 
\end{equation*}
satisfies the conditions of the Schr\"{o}der fixed-point theorem (see e.\,g. 
\cite{Collatz,OrtRhein}). 
  
There holds 
\begin{align*}
|\,(T(x))_i &- (T(y))_i |  \ 
 = \ \bigl(\,|\,T(x) - T(y)|\,\bigr)_i \\[10pt]
&= \ \Bigl(\;\bigl|\,\eus{V}\;\sti\,(\,\mbf{b}\ominus\mbf{H}
  \,\sti^{-1}(x)) - \eus{V}\;\sti\,(\,\mbf{b}\ominus\mbf{H}
  \,\sti^{-1}(y))\bigr|\;\Bigr)_i \\[10pt]
&= \ \Bigl(\;\bigl|\,\eus{V}\,\bigl(\,\sti\,(\,\mbf{b}\ominus\mbf{H}
  \,\sti^{-1}(x)) - \sti(\,\mbf{b}\ominus\mbf{H}\,\sti^{-1}(y))\bigr)\,\bigr|
  \;\Bigr)_i\\[10pt]
&= \ \Bigl(\;\bigl|\,\eus{V}\;\sti\,(\,\mbf{H}\,\sti^{-1}(x)
  \ominus\mbf{H}\,\sti^{-1}(y))\,\bigr|\;\Bigr)_i \\[10pt]
&\leq \ \Bigl(\,|\eus{V}|\cdot|\sti(\,\mbf{H}\,\sti^{-1}(x)
 \ominus\mbf{H}\,\sti^{-1}(y))|\,\Bigr)_i \\[10pt]
&= \ \left(\;|\eus{V}|\cdot\left(
 \begin{array}{c}
 \bigl|\,\un{\mbf{H}\,\sti^{-1}(x)\ominus\mbf{H}\,\sti^{-1}(y)}
  \,\bigr|\\[7pt]
 \bigl|\,\ov{\mbf{H}\,\sti^{-1}(x)\ominus\mbf{H}\,\sti^{-1}(y)}
  \,\bigr|
 \end{array}
 \right)\right)_i \\[10pt]
&\leq \ \left(\;|\eus{V}|\cdot\left(
 \begin{array}{l}
 \max\bigl\{\;\bigl|\,\un{\mbf{H}\,\sti^{-1}(x)
          \ominus\mbf{H}\,\sti^{-1}(y)}\,\bigr|, \\[5pt] 
 \phantom{MMMMM}\bigl|\,\ov{\mbf{H}\,\sti^{-1}(x)
          \ominus\mbf{H}\,\sti^{-1}(y)}\,\bigr|\;\bigr\} \\[7pt]
 \max\bigl\{\;\bigl|\,\un{\mbf{H}\,\sti^{-1}(x)
          \ominus\mbf{H}\,\sti^{-1}(y)}\,\bigr|, \\[5pt]
 \phantom{MMMMM}\bigl|\,\ov{\mbf{H}\,\sti^{-1}(x)
          \ominus\mbf{H}\,\sti^{-1}(y)}\,\bigr|\;\bigr\}
 \end{array}
 \right)\right)_{i} \\[10pt]
&= \ \left(\;|\eus{V}|\cdot\left(
 \begin{array}{c}
 \bigl|\,\Dist(\mbf{H}\,\sti^{-1}(x), \mbf{H}\,\sti^{-1}(y))\,\bigr|\\[7pt]
 \bigl|\,\Dist(\mbf{H}\,\sti^{-1}(x), \mbf{H}\,\sti^{-1}(y))
  \,\bigr|
 \end{array}
 \right)\right)_{i}\,.
\end{align*}
  
Using inequality \eqref{MatrLipschIneq}, we can continue our calculations as follows: 
\begin{align*}
|\,(T(x))_i - (T(y))_i| \ 
&\leq \ \left(\;|\eus{V}|\cdot\left(
 \begin{array}{c}
 |\mbf{H}|\cdot\Dist(\,\sti^{-1}(x), \sti^{-1}(y)\,)\\[5pt]
 |\mbf{H}|\cdot\Dist(\,\sti^{-1}(x), \sti^{-1}(y)\,)
 \end{array}
 \right)\right)_i \\[8pt]
&= \ \left(\;|\eus{V}|\,
     \begin{pmatrix}
     |\mbf{H}| & 0 \\[2pt]
     0 & |\mbf{H}|
     \end{pmatrix}
    \left(
 \begin{array}{c}
 \Dist(\,\sti^{-1}(x), \sti^{-1}(y)\,)\\[5pt]
 \Dist(\,\sti^{-1}(x), \sti^{-1}(y)\,)
 \end{array}
    \right)\right)_i \\[8pt]
&= \ \left(\;|\eus{V}|\,|\mbf{H}|\tl \left(
 \begin{array}{c}
 \Dist(\,\sti^{-1}(x), \sti^{-1}(y)\,)\\[5pt]
 \Dist(\,\sti^{-1}(x), \sti^{-1}(y)\,)
 \end{array}
    \right)\right)_i \\[8pt]
&= \ \bigl(\;|\eus{V}|\,|\mbf{H}|\tl \zeta(x,y)\;\bigr)_i \\[8pt]
&= \ \bigl(\;\mbox{$i$-th row of the matrix}\;|\eus{V}|\,|\mbf{H}|\tl 
    \bigr)\cdot\zeta(x,y).
\end{align*}
Therefore, 
\begin{align*}
\max & \{\,|\,(T(x))_i - (T(y))_i|, 
   \,|\,(T(x))_{i+n} - (T(y))_{i+n}|\,\} \\[12pt]
&= \ \max\bigl\{\,\bigl(|\eus{V}|\,|\mbf{H}|\tl\; \zeta(x,y)\bigr)_i,\,
 \bigl(|\eus{V}|\,|\mbf{H}|\tl\; \zeta(x,y)\bigr)_{i+n}\,\bigr\}\\[12pt]
&= \ \max\left\{\;\left(
 \mbox{\begin{tabular}{c}
         $i$-th row    \\
         of the matrix \\
       $|\eus{V}|\,|\mbf{H}|\tl$
       \end{tabular}}
 \right) \zeta(x,y),  \  \rule[-23pt]{0mm}{50pt}
 \left(
 \mbox{\begin{tabular}{c}
       $(i+n)$-th row \\
       of the matrix  \\
       $|\eus{V}|\,|\mbf{H}|\tl$
       \end{tabular}}
 \right) \zeta(x,y)\;
        \right\}.
\end{align*}
From the Frobenius formulas for inverting block matrices (see, for example, 
\cite{Gantmacher}), it follows that the $2n\times 2n$-matrix $\eus{V}$ is a block 
matrix of the same structure as the extended multiplier matrix $G\tl$, i.\,e., it is 
divided into four $n\times n$ blocks, and the diagonal blocks are the same. Therefore, 
in general, we obtain 
\begin{equation*}
\zeta\bigl(\,T(x), T(y)\,\bigr) \ \leq  \ 
|\eus{V}|\,|\mbf{H}|\tl\;\zeta(x, y),
\end{equation*}
as required. 
  
It is easy to see that the proof conducted is easily adaptable for the iterative 
process \eqref{2ndForm} as well, since $|\eus{H}(\mbf{x})| = |\mbf{Hx}|$ 
for any $\mbf{x}\in\mbb{KR}^n$. 
\end{proof}

\bigskip
The main concern of the developers of iterative methods of the form \eqref{Iterations} 
is to reduce, as much as possible, the spectral radius (or norm) of the so-called 
Lipschitz operator for the transition operator $\mbf{T}$, in order, first, to ensure 
the convergence of iterations, and secondly, to accelerate this convergence where it 
already exists. As follows from the proof of Theorem~\arabic{ARMSplitTheo}, the matrix 
of this Lipschitz operator is $|\eus{V}|\,|\mbf{H}|\tl$ for the iterations \eqref{1stForm}.  Optimization of the distributive splitting of the matrix $\mbf{A}$ into $G$ and $\mbf{H}$ 
is an interesting, but difficult task, and here we will not consider its solution in 
the most general form. We only note that splitting a point addend is especially 
convenient in practice for cases where the matrix of the interval systems has many 
point (non-interval) elements, while the proportion of substantially interval elements 
in the matrix is small.

  
\subsection{Triangular splitting of the matrix of the equation system}

With a triangular splitting of the interval matrix $\mbf{A}$, the identity 
\begin{equation*} 
\mbf{Ax} = \mbf{Gx} + \mbf{Hx}, 
  \qquad \mbf{x}\in\mbb{KR}^{n}, 
\end{equation*} 
is achieved due to the fact that the matrices $\mbf{G}$ and $\mbf{H}$ form a disjoint 
decomposition for $\mbf{A}$, since nonzero elements in $\mbf{G}$ and $\mbf{H}$ are 
mutually exclusive. The pseudocode of the iteration process \eqref{SplitForm} in 
$\mbb{KR}^n$, with the triangular splitting of the matrix $\mbf{A}$, which we call 
\texttt{TrnSplit}, has the form presented in Table~\ref{TrnSplitAlgo}, where 
``$\oslash$'' is an internal division in $\mbb{KR}$, i.\,e. multiplication by 
the inverse interval (if it exists). 
  
It is worth noting that the iterative method with the triangular splitting of the form 
we used, is called the Gauss-Seidel method for ordinary linear systems (see e.\,g. 
\cite{Varga}). It would be logical to call our algorithm the ``interval Gauss-Seidel 
method'', but this name has already been fixed with respect to another iterative 
process (see e.\,g. \cite{Neumaier,SharyBook}). Therefore, we call our method simply 
a~``method based on triangular splitting''. 
  
  
\begin{table}[hp]
\begin{center}
\caption{Algorithm \texttt{TrnSplit} for computing formal solutions,}
based on triangular splitting of the system matrix. \\[7pt]
\label{TrnSplitAlgo}
\tabcolsep 7mm
\begin{tabular}{|l|}
\hline\\
\hspace*{44mm}\textsf{Input}\\[5pt]
An interval linear algebraic $n\times n$-system $\mbf{A}x = \mbf{b}$.\\[3pt]
A triangular splitting of the matrix $\mbf{A}$ of the system \\[1pt] 
\phantom{I} to interval matrices $\mbf{G} = (\,\mbf{g}_{ij})$ 
  and $\mbf{H} = (\,\mbf{h}_{ij})$.\\[3pt]
A specified accuracy $\epsilon$.\\[10pt]
\hline\\
\hspace{43mm}\textsf{Output}\\[5pt]
An approximation to formal solution of the system $\mbf{A}x = \mbf{b}$.\\[10pt]
\hline\\
\hspace*{40mm}\textsf{Algorithm}\\
\parbox{84mm}{%
\begin{tabbing}
\= AAA\= AAA\= \kill
\> $q\gets +\,\infty$;\\[2pt]
\> assign a starting value to the vector $\mbf{x}$;\\[4pt]
\> \texttt{DO WHILE} \  ( $q\geq\epsilon$ )\\[2pt]
\>\> $\mbf{p}_1 \gets\mbf{b}_1$;\\[4pt]
\>\> \texttt{DO FOR} \ $i=2$ \ \texttt{TO} \ $n$\\[4pt]
\>\>\>
\begin{math}
\displaystyle
\mbf{p}_i \gets\mbf{b}_i \ominus\sum_{j=1}^{i-1}
  \mbf{h}_{ij}\mbf{x}_j
\end{math}\\[3pt]
\>\> \texttt{END DO}\\[2pt]
\>\> $\tilde{\mbf{x}}_{n} \gets\mbf{p}_n \oslash\mbf{g}_{nn}$;\\[4pt]
\>\> \texttt{DO FOR} \ $i=n-1$ \ \texttt{TO} \ $1$ 
     \  \texttt{STEP} \ $(-1)$\\[3pt]
\>\>\>
\begin{math}
\displaystyle
\tilde{\mbf{x}}_i \gets\Biggl(\;\mbf{p}_i \ominus\sum_{j=i+1}^n
  \mbf{g}_{ij} \tilde{\mbf{x}}_j \Biggr)\oslash\mbf{g}_{ii}
\end{math}\\[3pt]
\>\> \texttt{END DO}\\[2pt]
\>\> $q\gets$ distance between the vectors $\mbf{x}$ and $\tilde{\mbf{x}}$;\\[2pt]
\>\> $\mbf{x}\gets\tilde{\mbf{x}}$;\\[2pt]
\> \texttt{END DO}
\end{tabbing}}\\[5pt]
\hline
\end{tabular}
\end{center}
\end{table}
  
  
The study of the convergence of the algorithm \texttt{TrnSplit} was carried out 
by A.Yu.\,Karlyuk under the guidance of the author in \cite{Karlyuk}, and the main 
result of this article was as follows: 
  
\begin{theorem} 
For the interval matrix $\mbf{A}$ of the equation system \eqref{InLinSys2}, 
let the point $n\times n$-matrices $D$, $L$, $R$ are defined by the formulas 
\begin{eqnarray*} 
D &=& \diag \{\,|\,\inv\mbf{a}_{11}|, |\,\inv\mbf{a}_{22}|, \ldots , 
   |\,\inv\mbf{a}_{nn}|\,\} ,\\[7pt] 
L &=& (\,l_{ij}), \quad\mbox{ with }\quad 
   l_{ij} = \left\{ 
   \begin{array}{rl} 
   |\,\mbf{a}_{ij}|, & \mbox{ if } \ i>j,\\[2pt] 
    0, & \mbox{ if } \ i\le j,\\ 
   \end{array} 
   \right.\\[7pt] 
R &=& (\,r_{ij}), \quad\mbox{ with }\quad 
   r_{ij} = 
   \left\{ 
   \begin{array}{rl}
     0, & \mbox{ if } \ i\ge j,\\[2pt] 
   |\,\mbf{a}_{ij}|, & \mbox{ if } \ i<j, \\ 
   \end{array} 
   \right. 
\end{eqnarray*} 
where ``inv'' means taking the inverse element with respect to multiplication 
in Kaucher interval arithmetic. If the matrix 
\begin{equation*}
\eus{Q}\; := \;\sum_{j=0}^{n-1}(DL)^j DR \; = \;(I-DL)^{-1}DR
\end{equation*} 
satisfies $\rho(\eus{Q}) < 1$, then the iteration process {\rm\texttt{TrnSplit}} 
converges to a unique formal solution~$\mbf{x}^\ast$ of system \eqref{InLinSys2} 
from any starting approximation $\mbf{x}^{(0)}$. Additionally, the following estimate 
is valid 
\begin{equation*}
\Dist \bigl(\,\mbf{x}^{\ast},\,\mbf{x}^{(k)}\bigr) \  \leq  \ 
   \Biggl(\bigl(I-\eus{Q}\bigr)^{-1} -\sum_{j=0}^{k-1} \eus{Q}^j \Biggr) 
   \cdot \Dist \bigl(\mbf{x}^{(0)},\,\mbf{x}^{(1)} \bigr), 
\end{equation*}
where the multimetrics ``{\,\rm Dist}'' is defined in \eqref{VMultiMetric}. 
\end{theorem}

For $\rho\,(\eus{Q}) <1$ to take place, it suffices, for example, to satisfy 
the following condition on the interval matrix $\mbf{A}$ of system \eqref{InLinSys2}: 
recurrently calculated numbers $s_1$, $s_2$, \ldots, $s_n$, such that 
\begin{equation}
\label{SiLeqOne}
s_i =
\frac{1}{\langle\,\pro\mbf{a}_{ii}\rangle}\left(\:\sum_{j=1}^{i-1}
\,|\,\mbf{a}_{ij}|\,s_j +\sum_{j=i+1}^n |\,\mbf{a}_{ij}| \right),
\qquad i = 1, 2, \ldots, n ,
\end{equation}
are all strictly less than one. In turn, conditions \eqref{SiLeqOne} are certainly 
fulfilled for interval matrices with the property of diagonal dominance: 
\begin{equation*}
\langle\,\pro\mbf{a}_{ii}\rangle > \sum_{j\ne i}|\,\mbf{a}_{ij}|
\qquad
\text{for every} \  i = 1, 2, \ldots,n. 
\end{equation*}
  
A classical result of R.\,Varga \cite{Varga} states that, in the point case of 
$\mbf{A} = G - H = C$, $\mbf{b} = d$, the method \eqref{SplitForm} converges
the faster the smaller $H$ is.

   
\section{Numerical examples} 
\label{StatIterTests}

\vspace{-\medskipamount}  
\begin{example} 
Let us consider the interval linear system by Barth-Nuding \cite{BarNud} 
\begin{equation*} 
\tag{\ref{BarthNudingSystem}}
\begin{pmatrix} 
 [2, 4] & [-2, 1] \\[2mm] 
[-1, 2] &  [2, 4]
\end{pmatrix}
\,x =
\begin{pmatrix}
[-2, 2] \\[2mm] [-2, 2]
\end{pmatrix}. 
\end{equation*}
The algorithm \texttt{ARMSplit} after 10 iterations gives 3 valid significant digits 
of the exact answer $(-\frac{1}{3}, \frac{1}{3})^\top$, and after 20 iterations it 
gives 6 correct significant digits, which in order of labor costs is comparable 
to the iterative method from \cite{Zyuzin-87,Zyuzin-89,Zyuzin-90}, based on 
the cleavage of the main diagonal of the matrix of the interval system. The same 
figures are achieved by the algorithm \texttt{ARMSplit} when finding a formal 
solution of system \eqref{BarthNudingSystem} with the dualized matrix (which 
occurs, for example, in the internal evaluation of the united solution set 
for \eqref{BarthNudingSystem}). 
\end{example}

\begin{example} 
Consider an interval linear $40\times 40$-system with the matrix 
\begin{equation}
\label{ICDiffMatrix}
\arraycolsep 1pt
{\fontsize{9pt}{11pt}\selectfont
\left(\begin{array}{cccccc}
 [1.8, 2.2]  & [-1.1, -0.9] & & & \raisebox{-2.5ex}{\Huge 0} & \\[1mm]
[-1.1, -0.9] &  [1.8, 2.2 ] & [-1.1, -0.9] & & & \\[1mm]
 & [-1.1, -0.9] & [1.8, 2.2] & \ddots & & \\[1mm]
 & & \ddots & \ddots & \ddots & \\[1mm]
 &  &  & \ddots & [1.8, 2.2]  & [-1.1, -0.9]\\[1mm]
 & \mbox{\Huge 0} & & & [-1.1, -0.9] &  [1.8, 2.2]
\end{array}\right)}
\end{equation}
and right-hand side vector 
\begin{equation}
\label{RHSVector1}
\left(
\begin{array}{c}
[0.9, 1.1]\\[1mm] [1.8, 2.2]\\[1mm] [2.7, 3.3]\\[1mm] \vdots\\[1mm]
  [35.1, 42.9]\\[1mm] [36, 44]
\end{array}
\right). 
\end{equation}
  
\noindent
The matrix \eqref{ICDiffMatrix} is obtained from a popular tridiagonal matrix 
approximating the second derivative on a uniform grid by 10\% broadening of the elements, 
and the right-hand side of the system is obtained by the same broadening of the vector 
$(\,1, 2, 3, \ldots, 39, 40\,)^\top$. As in the case of the subdifferential Newton method, 
neither this system of equations nor the system derived from it by matrix dualization 
represent a serious problem for the algorithms presented in this paper. The method 
\text{ARMSplit} based on point splitting finds 12--13 correct significant digits for 
the endpoints of the components of formal solutions of both the original interval 
system and the system with a dualized matrix after 16 iterations. 
\end{example}

\begin{example}
Consider an interval linear $40\times 40$-system with Neumaier matrix 
\cite{Neumaier} having the number $40$ at the main diagonal, i.\,e. 
\begin{equation}
\label{NeumaierMatrix40}
\left(  \
\begin{array}{ccccc}
   40    & [ 0, 2 ] & \cdots & [ 0, 2 ]  &  [ 0, 2 ] \\[2mm]
[ 0, 2 ] &    40    & \cdots & [ 0, 2 ]  &  [ 0, 2 ] \\[2mm]
 \vdots  &  \vdots  & \ddots &  \vdots   &   \vdots  \\[2mm]
[ 0, 2 ] & [ 0, 2 ] & \cdots &    40     &  [ 0, 2 ] \\[2mm]
[ 0, 2 ] & [ 0, 2 ] & \cdots & [ 0, 2 ]  &     40
\end{array}
\  \right), 
\end{equation}
and the right-hand side vector 
\begin{equation}
\label{RHSVector2}
\left(
\begin{array}{c}
[10, 20] \\[1mm]
[10, 20] \\[1mm]
 \vdots  \\[1mm]
[10, 20]
\end{array}
\right). 
\end{equation}
The algorithm \texttt{ARMSplit} computes, after 40 iterations, the approximation 
to the formal solution 
\begin{equation*}
\left(
\begin{array}{c}
[0.25, 0.16949152542]\\[1mm]
[0.25, 0.16949152542]\\[1mm]
  \vdots\\[1mm]
[0.25, 0.16949152542]
\end{array}
\right)
\end{equation*}
with an accuracy of about $10^{-8}$.
  
The same can be observed when calculating the formal solution to the interval linear 
system with the matrix dualized to \eqref{NeumaierMatrix40} and the right-hand side 
\eqref{RHSVector2}. Note that in this example, the interval matrix of the system is 
singular (see \cite{Neumaier}), but this does not prevent the algorithm 
\texttt{ARMSplit} from successfully finding a formal solution. 
\end{example}

\begin{example} 
For the interval linear $7\times 7$-system 
   
{\footnotesize
\begin{displaymath}
\arraycolsep 2pt
\left(\begin{array}{ccccccc}
 [4,6]  & [-9,0]  & [0,12]  & [2,3]  & [5,9]  & [-23,-9]& [15,23] \\[1mm]
 [0,1]  & [6,10]  & [-1,1]  & [-1,3] & [-5,1] & [1,15]  & [-3,-1] \\[1mm]
 [0,3]  & [-20,-9]& [12,77] &[-6,30] & [0,3]  & [-18,1] & [0,1]   \\[1mm]
 [-4,1] & [-1,1]  & [-3,1]  & [3,5]  & [5,9]  & [1,2]   & [1,4]   \\[1mm]
 [0,3]  & [0,6]   & [0,20]  & [-1,5] & [8,14] & [-6,1]  & [10,17] \\[1mm]
[-7,-2] & [1,2]   & [7,14]  & [-3,1] & [0,2]  & [3,5]   & [-2,1]  \\[1mm]
 [-1,5] & [-3,2]  & [0,8]   & [1,11] & [-5,10]& [2,7]   & [6,82]
\end{array}\right)\,\mbox{\normalsize$x$} =
\left(\begin{array}{c}
[-10,95]\\[1mm] [35,14]\\[1mm] [-6,2]\\[1mm] [30,7]\\[1mm] [4,95]\\[1mm]
[-6,46] \\[1mm] [-2,65]
\end{array}\right)
\end{displaymath}}
  
\noindent
from the work \cite{Shary-95}, the algorithm \texttt{ARMSplit} diverges, but its 
formal solution can be successfully found with the use of the subdifferential Newton 
method (after 9 iterations and for the damping factor $\tau = 1$). 
  
When the (7,7)-element of the interval matrix is narrowed, the convergence of 
the algorithm \texttt{ARMSplit} to the formal solution appears, but it is very slow. 
For example, with $\mbf{a}_{77} = [8, 82]$, the algorithm requires about a hundred 
iterations to get 5 correct significant digits of the answer. 
\end{example}

Summarizing the last example, we can say that it demonstrates the advantage of 
the subdifferential Newton method over stationary iterative methods not only 
in terms of efficiency, but also in terms of the applicability scope.


  
\end{document}